\documentclass[onefignum,onetabnum]{siamart190516}
\usepackage{subfigure,graphics,multirow,marginnote,enumerate,bm}
\usepackage{lipsum}
\usepackage{amsfonts}
\usepackage{graphicx}
\usepackage{epstopdf}
\usepackage{algorithmic}
\ifpdf
  \DeclareGraphicsExtensions{.eps,.pdf,.png,.jpg}
\else
  \DeclareGraphicsExtensions{.eps}
\fi

\newsiamremark{remark}{Remark}
\newsiamremark{hypothesis}{Hypothesis}
\crefname{hypothesis}{Hypothesis}{Hypotheses}
\newsiamthm{claim}{Claim}

\headers{Inverse diffraction grating problems}{X.Xu, G. Hu, B. Zhang, and H. Zhang}

\title{Uniqueness in inverse diffraction grating problems with infinitely many plane
waves at a fixed frequency \thanks{Submitted to the editors DATE.
\funding{The work of G. Hu is supported by the National Natural Science Foundation of China (No. 12071236),
the Fundamental Research Funds for Central Universities in China (No. 63213025)
and Beijing Natural Science Foundation (No. Z210001).
The work of X. Xu is supported by the China Postdoctoral Science Foundation Grant with No. 2019TQ0023 and the National Natural Science Foundation of China (No. 12122114).
The work of H. Zhang is supported by the National Natural Science Foundation of China (No. 11871466)
and Youth Innovation Promotion Association CAS.}}}

\author{Xiaoxu Xu\thanks{School of Mathematics and Statistics, Xi'an Jiaotong University, Xi'an, Shaanxi, 710049,
China
  (\email{xuxiaoxu@xjtu.edu.cn}).}
\and Guanghui Hu\thanks{(Corresponding author) School of Mathematical Sciences and LEMP, Nankai University, Tianjin 300071,
China
  (\email{ghhu@nankai.edu.cn}).}
\and
Bo Zhang\thanks{LSEC and Academy of Mathematics and Systems Science, Chinese Academy of
Sciences, Beijing 100190, China and School of Mathematical Sciences, University of Chinese
Academy of Sciences, Beijing 100049, China
  (\email{b.zhang@amt.ac.cn}).}
\and
Haiwen Zhang\thanks{Academy of Mathematics and Systems Science, Chinese Academy of Sciences,
Beijing 100190, China
  (\email{ zhanghaiwen@amss.ac.cn}).}
}

\usepackage{amsopn}

\ifpdf
\hypersetup{
  pdftitle={Uniqueness in inverse diffraction grating problems with infinitely many plane
waves at a fixed frequency},
  pdfauthor={X.Xu, G. Hu, B.Zhang and H. Zhang}
}
\fi

\newcommand{\R}{{\mathbb R}}

\newcommand{\Z}{{\mathbb Z}}
\newcommand{\N}{{\mathbb N}}
\newcommand{\C}{{\mathbb C}}
\newcommand{\D}{{\rm d}}

\newcommand{\no}{\nonumber}
\newcommand{\be}{\begin{eqnarray}}
\newcommand{\ben}{\begin{eqnarray*}}
\newcommand{\en}{\end{eqnarray}}
\newcommand{\enn}{\end{eqnarray*}}
\newcommand{\ba}{\backslash}
\newcommand{\pa}{\partial}

\newcommand{\ov}{\overline}

\newcommand{\G}{\Gamma}
\newcommand{\al}{\alpha}
\newcommand{\Om}{\Omega}

\newtheorem{thm}{Theorem}[section]
\newtheorem{cor}{Corollary}[section]
\newtheorem{lem}{Lemma}[section]
\newtheorem{prop}{Proposition}[section]
\newtheorem{defn}{Definition}[section]

\newtheorem{rem}{Remark}[section]
\definecolor{red}{rgb}{0,0,0}
\definecolor{rot}{rgb}{0,0,0}
\definecolor{hw}{rgb}{0,0,0}
\definecolor{blue}{rgb}{0,0,0}

\begin{document}

\maketitle

% REQUIRED
\begin{abstract}
This paper is concerned with the inverse diffraction problems by a periodic
curve with Dirichlet boundary condition in two dimensions.
It is proved that the periodic curve can be uniquely determined by the near-field measurement data
corresponding to infinitely many incident plane waves with distinct directions at a fixed frequency.
Our proof is based on Schiffer's idea which consists of two ingredients:
i) the total fields for incident plane waves with distinct directions are linearly independent, and
ii) there exist only finitely many linearly independent Dirichlet eigenfunctions in a bounded domain
or in a closed waveguide under additional assumptions on the waveguide boundary.
Based on the Rayleigh expansion, we prove that the phased near-field data can be uniquely determined
by the phaseless near-field data in a bounded domain, with the exception of a finite set of incident angles.
Such a phase retrieval result leads to a new uniqueness result for
the inverse grating diffraction problem with phaseless near-field data at a fixed frequency.
Since the incident direction determines the quasi-periodicity of the boundary value problem, our inverse issues are different from the existing results of [Htttlich \& Kirsch, Inverse Problems 13 (1997): 351-361] where fixed-direction plane waves at multiple frequencies were considered.
\end{abstract}

% REQUIRED
\begin{keywords}
uniqueness, grating diffraction problem,
Dirichlet boundary condition, phaseless data
\end{keywords}

% REQUIRED
\begin{AMS}
 35R30, 78A46, 35B27.
\end{AMS}

\section{Introduction}\label{sec:periodic}

Suppose a perfectly conducting grating is illuminated by an incident monochromatic plane wave in
an isotropic homogeneous background medium.
For simplicity it is assumed that the grating is periodic in one surface direction $x_1$ and independent
of another surface direction $x_3$.
In the present paper, we restrict the discussions to the TE polarization case, where
the three-dimensional scattering problem governed by the Maxwell equations can be reduced to a
two-dimensional diffraction problem modeled by the scalar Helmholtz equation over the $x_1x_2$-plane.
Accordingly, the perfect conductor boundary condition on the grating surface can be reduced to the
Dirichlet boundary condition.
This work is concerned with the inverse diffraction problem of recovering the periodic curve (i.e.,
the cross-section of the grating surface) with a Dirichlet boundary condition from phased and phaseless
near-field data measured above the grating.

Denote by $\G\!\subset\!\R^2$ a curve periodic in the $x_1$-direction and bounded in the $x_2$-direction
which represents the cross-section of the grating surface in the $x_1x_2$-plane.
Let the incident field be a time-harmonic plane wave of the form $u^i(x)e^{-i\omega t}$, incited at
the angular frequency $\omega>0$, where the spatially dependent function $u^i$ takes the form
\be\label{plane-wave}
u^i(x)=e^{ikx\cdot d}=e^{ikx_1\sin\theta-ikx_2\cos\theta},\quad x=(x_1,x_2)\in\R^2.
\en
Here the incident direction $d\!:=\!(\sin\theta,\!-\!\cos\theta)$ is given in terms of the incident angle $\theta\!\in\!(\!-\pi/2,\!\pi/2)$ and $k\!:=\!\omega/c$ is the wave number with $c>0$ denoting the wave
speed in the homogeneous background medium.
In this paper we assume further that $\G$ satisfies one of the following regularity conditions:

\textbf{Condition (i)} $\;$ $\Gamma$ is the graph of a 3-times continuously differentiable function;

\textbf{Condition (ii)} $\,\!$ $\Gamma$ is an analytical curve.

\noindent
Denote by $L>0$ the period of $\Gamma$ and by $\Omega$ the unbounded connected domain above $\G$
(cf. Figure \ref{pr}).
The wave propagation is then modelled by the Dirichlet boundary value problem for the Helmholtz equation
\be\label{equation}
\Delta u+k^2 u=0\quad\mbox{in}\;\;\Omega,\qquad u=0\quad\mbox{on}\;\;\G,
\en
where the total field $u=u^i+u^s$ is the sum of the incident field $u^i$ and the scattered field $u^s$.

\begin{figure}[htbp]
  \centering
%  \includegraphics[width=0.7\textwidth]{fig/pr.eps}
  \includegraphics[width=3in]{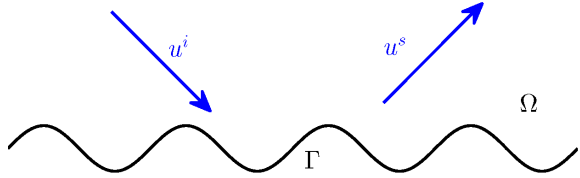}
  \caption{Scattering by a periodic curve with Dirichlet boundary condition.
  }\label{pr}
\end{figure}

Set $\alpha=\alpha(k,\theta):=k\sin\theta$.
Obviously, the incident field (\ref{plane-wave}) is $\alpha$-quasi-periodic in the sense
that $e^{-i\al x_1}u^i(x)$ is $L$-periodic with respect to $x_1$ for all $x\!\in\!\Om$.
In view of the periodicity of the structure together with the form of the incident field,
we require the total field $u$ to be $\al$-quasi-periodic, that is, $e^{-i\al x_1}u(x)$ is $L$-periodic
with respect to $x_1$ for all $x\in\Om$.  This implies that
\be\label{qp}
u(x_1+nL, x_2)=u(x_1, x_2)e^{i\alpha n L}\quad \mbox{for any}\quad n\in \Z.
\en
The number $\alpha\in\R$ will be referred to as the phase shift of the solution.
Since the domain $\Omega$ is unbounded in the $x_2$-direction, a radiation condition needs to be
imposed at infinity as $x_2\rightarrow\infty$ to ensure the well-posedness of the diffraction problem.
Precisely,
we require the scattered field $u^s$ to satisfy the Rayleigh expansion, that is, there exist
Rayleigh coefficients $A_n\in\C$ ($n\in\mathbb{Z}$) depending on $k$, $\theta$ and $\Gamma$ such that
\be\label{Rayleigh-up}
u^s(x)=\sum_{n\in\Z}A_ne^{i\alpha_n x_1+i\beta_n x_2},\quad x\in U_h:=\{x=(x_1,x_2)\in\R^2:x_2>h\}
\en
where the parameters $\alpha_n\in\R$ and $\beta_n\in\C$ for $n\in\Z$ are defined by
\be\label{beta}\begin{split}
&\alpha_n=\alpha_n(k,\theta,L):=\alpha+2n\pi/L,\\
&\beta_n=\beta_n(k,\theta,L):=
\left\{\begin{array}{lll}
\sqrt{k^2-(\al_n)^2}&&\mbox{if}\quad |\al_n|\le k,\\
i\sqrt{(\al_n)^2-k^2}&&\mbox{if}\quad |\al_n|>k,
\end{array}\right.
\end{split}
\en
for any fixed $h>\max\{x_2:x\in\G\}$.
We note that the series (\ref{Rayleigh-up}) is uniformly convergent and bounded in $U_h$
(see Lemma \ref{lem-x} below).
It consists of a finite number of propagating wave modes for $|\alpha_n|\!\leq\!k$ and infinitely
many surface (evanescent) wave modes corresponding to $|\alpha_n|\!>\!k$.
For notational convenience we rewrite the incident plane wave (\ref{plane-wave}) as
\be\label{ui}
u^i(x)=A_{\imath}e^{i\alpha_{\imath}x_1+i\beta_{\imath}x_2},
\en
where $A_{\imath}=A_{\imath}(k,\theta):=1$, $\alpha_{\imath}=\alpha_{\imath}(k,\theta):=k\sin\theta$, $\beta_{\imath}=\beta_{\imath}(k,\theta):=-k\cos\theta$.
Here, the symbol $\imath$ denotes the index for the incident plane wave.
We note that $\alpha_{\imath}=\alpha=\alpha_0$ and $\beta_{\imath}=-\beta_0$.

The well-posedness of the forward diffraction problem is presented in the following proposition.

\begin{prop}\label{Lem:planewave}
(1) If Condition (i) holds, the diffraction problem (\ref{plane-wave})--(\ref{Rayleigh-up}) admits
a unique $\alpha$-quasi-periodic solution $u\in C^2(\Om)\cap C(\overline{\Omega})$.

(2) Under Condition (ii), there exists at least one solution $u\in C^2(\Om)\cap C(\overline{\Omega})$
to the diffraction problem (\ref{plane-wave})--(\ref{Rayleigh-up}).
Moreover, uniqueness of the solution remains true for small wave numbers or for all wave numbers
excluding a discrete set with the only accumulating point at infinity.
\end{prop}

We refer to \cite{J.M2002,Kirsch1} for the proof of the first statement when the period of the curve
is $L=2\pi$. Actually, it follows from a scaling argument that the statement (1) holds for an arbitrary
period $L>0$. Further, by the Fredholm alternative (see, e.g., \cite[Theorem 2.33]{Monk}) and the
analytic Fredholm theory (see, e.g., \cite[Theorem 8.26]{CK19}), one can prove the second statement
through a standard variational argument together with quasi-periodic transparent boundary conditions
(see, e.g., \cite{AbNe,Tilo-hab,BBS94,B1}).
We remark that the well-posedness of the diffraction problem (\ref{plane-wave})--(\ref{Rayleigh-up})
can be established under weaker conditions than Conditions (i) and (ii).
To be more specific, if $\Gamma$ is a Lipschitz curve, the existence of $\al$-quasi-periodic variational
solutions in $H_{0,\alpha}^1(\Om)$ can be shown, where
\ben
H^1_{0,\alpha}(\Om):=\{u\in H^1_{loc}(\Om):\mbox{$e^{-i\alpha x_1}u(x)$ is $L$-periodic with respect to $x_1$},\;u=0\;\mbox{on}\;\G\}.
\enn
Further, uniqueness of solutions remains valid for any $k>0$ even under the following weaker assumption
(see \cite[(4.1) and Theorem 4.1]{Monk05} and \cite[(2.2) and Theorem 4.1]{CJ10}):
\ben
(x_1,x_2)\in\Omega\Rightarrow(x_1,x_2+s)\in\Omega\quad\mbox{for all}\;s>0.
\enn
Note that this geometric assumption is fulfilled if $\Gamma$ is the graph of a continuous function.

The inverse problem we consider in this paper is to recover a periodic curve with Dirichlet boundary
condition from phased or phaseless near-field data corresponding to an infinite number of incident plane
waves with different angles, where the period $L$ of the curve is unknown.

Let $\theta_n\in(-\pi/2,\pi/2)$ with $n\in\Z_+$ be distinct incident angles, and denote
by $u(x;\theta_n)$ the total field corresponding to the diffraction problem
(\ref{plane-wave})--(\ref{Rayleigh-up}) with $\theta=\theta_n$.
Note that, according to Proposition \ref{Lem:planewave}, the diffraction problem
(\ref{plane-wave})--(\ref{Rayleigh-up}) may admit multiple solutions under Condition (ii) if $k$
is an exceptional wavenumber. If this happens, $u(x;\theta_n)$ is assumed to any one of
these solutions. The main uniqueness result for the inverse problem considered is presented in the
following theorem.

\begin{theorem}\label{TH-main}
Assume that the unknown periodic curve $\Gamma$ with Dirichlet boundary condition satisfies
either Condition (i) or Condition (ii). Suppose the period of $\Gamma$ is unknown.
Then $\G$ can be uniquely determined by either the phased data
$\{u(x;\theta_n):x\in\mathcal{S}\}_{n=1}^\infty$,
where $\mathcal{S}\subset\G_h$ is a line segment parallel to the $x_1$-axis, or by the phaseless data $\{|u(x;\theta_n)|:x\in\mathcal{D}\}_{n=1}^\infty$, where $\mathcal{D}\subset\Omega$ is a bounded domain.
Here, $\G_h:=\{x:x_2=h\}$ with $h>\max\{x_2:x\in\G\}$ being an arbitrary constant.
\end{theorem}

The proof of Theorem \ref{TH-main} will be given in Section \ref{sec4} for the case of phased data and in
Section \ref{sec5} for the case of phaseless data. If the background medium is non-absorbing (i.e., $k>0$),
it is well known that the global uniqueness with phased near-field data corresponding to one incident plane
wave is impossible (see \cite{EH2010}).
We will show in Section \ref{sec3} that phased near-field data corresponding to one
incident plane wave cannot even determine the period of a grating curve.
To the best of our knowledge, uniqueness for one incident wave was verified in the following special cases:
\begin{itemize}
 \item[(i)] the background medium is lossy (i.e., ${\rm Im}\,k>0$) \cite{B94};
 \item[(ii)] the wave number or the grating height is sufficiently small \cite{HF97};
 \item[(iii)] within the class of rectangular gratings \cite{ESY}, or within the class of polygonal gratings
 in the case that Rayleigh frequencies are excluded (i.e., $\beta_n\neq 0$ for all $n\in\Z$) \cite{EH2010}.
\end{itemize}
If a Rayleigh frequency occurs (i.e., $\beta_n=0$ for some $n\in\Z$), the measured data for two incident
plane waves can be used to determine a general polygonal grating \cite{ESY} (see also \cite{B2010,BZZ}
in the case of inverse electromagnetic scattering from perfectly conducting polyhedral gratings).
It was proved in  \cite{K94}%\cite{K94,YZ11a,YZ11b}
that a general periodic curve can be uniquely determined by using all
$\alpha$-quasi-periodic incident waves $\{e^{i\alpha_n x_1-i\beta_n x_2}:n\in\Z\}$.
Note that such kind of incident waves include a finite number plane waves for $|\alpha_n|\le k$ and
infinitely many evanescent waves corresponding to $|\alpha_n|>k$.
The factorization method established in \cite{Arens2005} also gives rise to the same uniqueness result.
If the a priori information of the grating height is available, Hettlich and Kirsch \cite{HF97} obtained
a uniqueness result by using fixed-direction plane waves with a finite number of frequencies.
This can be viewed as an extension of the idea due to Colton and Sleeman \cite{CS1983} from the case
of inverse scattering by bounded sound-soft obstacles to the case of inverse scattering by periodic structures. As will be seen in subsection \ref{sub-case2},  the fixed-direction problem of \cite{HF97} and the fixed-frequency problem to be investiagted here result in different eigenvalue problems. Using different directions leads to a $\mu$-eigenvalue problem where $\mu=\sin\theta$ is determined by the incidient angle $\theta\in(-\pi/2, \pi/2)$, which brings difficulites in proving the discreteness of  eigenvalues. To apply the analytical Fredholm theory,  we shall resort on the arguments of \cite{SW2002} to exclude the existence of flat dispersion curves in a closed waveguide.

In many practical applications, it is difficult to accurately measure the phase information of wave fields.
This motivates us to study the inverse problem of whether it is possible to recover a periodic curve with
Dirichlet boundary condition from phaseless data.
However, most uniqueness results with phaseless data are confined to inverse scattering from bounded
scatterers (see, e.g., \cite{JLZ19b,K14,K17,K17x,N15,XZZ18a}).
%We refer to \cite{K17,K17x} for uniqueness results to inverse medium scattering problem,
%which were established by the high-frequency asymptotic behavior of the scattered field.
In particular, using the decaying property of the scattered field at infinity, explicit formulas for
recovering phased far-field pattern from phaseless near-field data are derived in \cite{N15}.
In this paper we also prove a phase retrieval result but based on the Rayleigh expansion (\ref{Rayleigh-up})
for diffraction grating problems.
To the best of our knowledge, uniqueness results for identifying periodic grating curves using phaseless
near-field data are not available so far. We refer to \cite{AHN,ACZ,BLL,JLZ19b,K18,XZZ,ZZ17,ZZ18}
for numerical schemes to inverse scattering using phaseless data.

This paper is organized as follows.
In Section \ref{sec2}, we prepare several lemmas for later use.
Section \ref{sec3} is devoted to determining one grating period from the phased near-field data
for one incident plane wave. The results in Sections \ref{sec2} and \ref{sec3} are independent of
the smoothness Conditions (i) and (ii) of the periodic curve made in the introduction part.
In Section \ref{sec4}, we prove uniqueness for recovering periodic curves with Dirichlet boundary condition
using the phased near-field data corresponding to infinitely many incident plane waves with distinct directions.
A similar uniqueness result based on phaseless near-field data will be established in Section \ref{sec5}.
Finally, concluding remarks will be given in Section \ref{sec6}.

\section{Preliminary lemmata}\label{sec2}
\setcounter{equation}{0}

The following lemmas are useful in the proofs of uniqueness results in the sequel.

\begin{lem}\label{lem-x}
Let $\G$ be a periodic curve.
Set $U_h:=\{x\in\R^2\!:\!x_2>h\}$ for any $h>\max\{x_2\!:\!x\in\G\}$.

(i) The Rayleigh expansion (\ref{Rayleigh-up}) is uniformly bounded for $x\in\!U_h$.

(ii) The Rayleigh expansion (\ref{Rayleigh-up}) is uniformly and absolutely convergent for $x\in\!U_h$.

(iii) Let $b\in\mathbb{R}$ and let $A_n$ ($n\in\mathbb{Z}$) be given as in (\ref{Rayleigh-up}).
Set $\mathcal P_{\pm}(N):=\{n\in\Z:|\alpha_n|>k,\pm n>N\}$ for $N>0$.
Then, for the case when $b<|\beta_n|$ for all $n\in\mathcal P_{+}(N)$ the series
\be\label{1119-1}
\sum_{n\in\mathcal P_{+}(N)}A_ne^{i\alpha_n x_1+i\beta_n x_2+bx_2}
\en
is uniformly and absolutely convergent for $x\in\!U_h$.
For the case when $b<|\beta_n|$ for all $n\in\mathcal P_{-}(N)$,
the series
\ben%\label{1119-2}
\sum_{n\in\mathcal P_{-}(N)}A_ne^{i\alpha_n x_1+i\beta_n x_2+bx_2}
\enn
is uniformly and absolutely convergent for $x\!\in\!U_h$.

(iv) Let $N\in\Z_+$, $a_j\in\C$ and $b_j\in\R\ba\{0\}$ for $j=1,\ldots,N$.
Then
\ben
\left|\frac1T\int_T^{2T}\sum_{j=1}^Na_je^{ib_jt}dt\right|\leq \frac{2}{T}
\sum_{j=1}^N\frac{|a_j|}{|b_j|}\rightarrow 0\quad \textrm{as}~T\rightarrow+\infty.
\enn
\end{lem}

\begin{proof}
Choosing $\sigma>0$ small enough so that $h-2\sigma>\max\{x_2:x\in\G\}$, noting
that (\ref{Rayleigh-up}) also holds with $h$ replaced by $h-2\sigma$ and applying Parseval's
equality yield the estimate
\be\no
2|u^s(x)|&\leq&\sum_{n\in\Z}2\left|A_ne^{i\alpha_nx_1+i\beta_nx_2}\right|\\ \no
&\leq&\sum_{n\in\Z}\left|A_ne^{i\alpha_nx_1+i\beta_n(h-\sigma)}\right|^2+\sum_{n\in\Z}
\left|e^{i\beta_n(x_2-h+\sigma)}\right|^2\\ \label{211108-1}
&\leq&\frac1{L}\int_0^{L}|u^s(x_1,h-\sigma)|^2dx_1+\sum_{|\alpha_n|\leq k}1+\sum_{|\alpha_n|>k}Ce^{-|n|/C}
\en
uniformly for all $x\!\in\!U_h$, where we have used the fact that
$\sigma\sqrt{(2n\pi/L+\alpha_0)^2-k^2}>|n|/C$ holds for sufficiently large $|n|$ provided
the constant $C>0$ is large enough.
Thus statement (i) holds. The estimate (\ref{211108-1}) also implies that statement (ii) holds.

We now prove statement (iii).
We only consider the case when $b<|\beta_n|$ for all $n\in\mathcal P_{+}(N)$ since
the proof of the other cases is similar. We first conclude from (\ref{211108-1})
that $\{|A_ne^{i\alpha_nx_1+i\beta_n(h-\sigma)}|:n\in\mathcal P_{+}(N)\}$ is uniformly bounded.
Noting that in this case $(i\beta_n+b)<0$ for all $n\in\mathcal P_{+}(N)$, we have
\ben
\sum_{n\in\mathcal P_{+}(N)}\left|A_ne^{i\alpha_n x_1+i\beta_n x_2+bx_2}\right|
&\leq&\sum_{n\in\mathcal P_{+}(N)}\left|A_ne^{i\alpha_nx_1+i\beta_n(h-\sigma)}\right|
e^{b(h-\sigma)}e^{(i\beta_n+b)(x_2-h+\sigma)}\\
&\leq&\sum_{n\in\mathcal P_{+}(N)}Ce^{-|n|/C}
\enn
uniformly for all $x\in\!U_h$, where we have used the fact that
$\sigma\sqrt{(2n\pi/L+\alpha_0)^2-k^2}-b>|n|/C$ holds for sufficiently large $|n|$ provided
the constant $C>0$ is large enough.
This implies that (\ref{1119-1}) is uniformly and absolutely convergent for $x\in\!U_h$.

Finally, noting that
\ben
\left|\frac1T\int_T^{2T}a_je^{ib_jt}dt\right|=\left|\frac1T\frac{a_j(e^{2ib_jT}-e^{ib_jT})}{ib_j}\right|
\leq\frac1T\frac{2|a_j|}{|b_j|},\quad j=1,\ldots,N,
\enn
it is easy to see that statement (iv) holds.
\end{proof}

\begin{lem}\label{lem-x2}
Let $u(x;\theta_m)$ be the total field corresponding to the diffraction problem
(\ref{plane-wave})--(\ref{Rayleigh-up}) with the incident angle $\theta=\theta_m\in(-\pi/2,\pi/2)$
for $m=1,\ldots,M$ and $M\in\Z_+$.
Suppose $\{\theta_{m}\}_{m=1}^M$ are distinct incident angles.
Then $\{u(x;\theta_m)\}_{m=1}^M$ are linearly independent in $\Omega$.
\end{lem}

\begin{proof}
Assume that $\sum^{M}_{m=1}c_mu(x;\theta_m)=0$ in $\Omega$ for some $c_m\in\C$, $m=1,\ldots,M$.
To indicate the dependence of $u^s$ on the incident angle, we rewrite the Rayleigh expansion
(\ref{Rayleigh-up}) as
\ben
u^s(x;\theta_m)=\sum_{n\in \Z} A_{n}(\theta_m)\;e^{i\alpha_n(\theta_m)x_1+i\beta_n(\theta_m)x_2},\;x\in U_h,
\enn
where $h>\max\{x_2:x\in\G\}$ and $\alpha_n(\theta_m):=\alpha(\theta_m)+2n\pi/L$ with
$\alpha(\theta_m):=k\sin\theta_m$ and $\beta_n(\theta_m)\in\C$ are defined as in (\ref{beta})
with the incident angle $\theta=\theta_m$. Then, by (\ref{ui}) it follows that
\be\label{lin-indep:eq1}
\sum\limits^{M}_{m=1}c_m\,u(x;\theta_m)=\sum^{M}_{m=1}c_m
\left(\sum_{n\in\Z\cup\{\imath\}}A_n(\theta_m)\;
e^{i\alpha_n(\theta_m)x_1+i\beta_n(\theta_m)x_2}\right)=0.
\en
For any $\tilde m\!\in\!\{1,2,\ldots,M\}$, multiplying (\ref{lin-indep:eq1}) by
$e^{-i\beta_{\imath}(\theta_{\tilde m})x_2}$ we obtain
\be\no
&&\sum_{m\in\mathcal I_{\tilde m}}c_me^{i\alpha_{\imath}(\theta_m)x_1}+
\sum_{m\in\{1,\ldots,M\}\ba\mathcal I_{\tilde m}}c_me^{i\alpha_{\imath}(\theta_m)x_1
+i[\beta_{\imath}(\theta_m)-\beta_{\imath}(\theta_{\tilde m})]x_2}\\ \label{1119+1}
&&+\sum^{M}_{m=1}c_m
\left(\sum_{n\in\Z}A_n(\theta_m)\;e^{i\alpha_n(\theta_m)x_1
+i[\beta_n(\theta_m)-\beta_{\imath}(\theta_{\tilde m})]x_2}\right)=0,~x\in U_h,
\en
where $\mathcal I_{\tilde m}:=\{m\in\{1,\ldots,M\}:\beta_{\imath}(\theta_m)=\beta_{\imath}(\theta_{\tilde m})\}$.

Next we claim that
\be\label{1119+2}
\lim_{H\rightarrow{\color{hw}+\infty}}\frac1H\int_{H}^{2H}\sum_{m\in\{1,\ldots,M\}\ba\mathcal I_{\tilde m}}
c_me^{i\alpha_{\imath}(\theta_m)x_1+i[\beta_{\imath}(\theta_m)
-\beta_{\imath}(\theta_{\tilde m})]x_2}dx_2=0,\\ \label{1119+3}
\lim_{H\rightarrow{\color{hw}+\infty}}\frac1H\int_{H}^{2H}\sum^{M}_{m=1}c_m
\left(\sum_{n\in\Z}A_n(\theta_m)\;e^{i\alpha_n(\theta_m)x_1+i[\beta_n(\theta_m)
-\beta_{\imath}(\theta_{\tilde m})]x_2}\right)dx_2=0
\en
for all $x_1\!\in\!\mathbb{R}$.
In fact, (\ref{1119+2}) {\color{hw}follows easily} from Lemma \ref{lem-x} (iv).
To prove (\ref{1119+3}), let $m\!\in\!\{1,\ldots,M\}$ be arbitrarily fixed.
For $N\!>\!0$ large enough we set $\mathcal J_1(N):=\{n\in\Z:|\alpha_n(\theta_m)|>k,|n|>N\}$
and $\mathcal J_2(N):=\{n\in\Z:|\alpha_n(\theta_m)|>k,|n|\leq N\}$.
Using $|e^{-i\beta_{\imath}(\theta_{\tilde m})x_2}|\!=\!1$, it follows from Lemma \ref{lem-x} (ii) that
\be\label{eq3}
\lim_{N\rightarrow+\infty}\sum_{n\in\mathcal J_1(N)}
\left|A_n(\theta_m)\;e^{i\alpha_n(\theta_m)x_1
+i[\beta_n(\theta_m)-\beta_{\imath}(\theta_{\tilde m})]x_2}\right|=0
\en
uniformly for all $x\in U_h$.
For any fixed $N\in\Z_+$, since $\mathcal J_2(N)$ is a finite set and
$i\beta_n(\theta_m)<0$ for all $n\in\mathcal J_2(N)$,
we have
\be\label{eq8}
\lim_{x_2\rightarrow+\infty}\sum_{n\in\mathcal J_2(N)}
\left|A_n(\theta_m)\;e^{i\alpha_n(\theta_m)x_1
+i[\beta_n(\theta_m)-\beta_{\imath}(\theta_{\tilde m})]x_2}\right|=0
\en
uniformly for all $x_1\in\mathbb{R}$.
Since $\mathcal J_3:=\{n\in\Z:|\alpha_n(\theta_m)|\le\!k\}$ is also a finite set and $\beta_n(\theta_m)\ge0>\beta_{\imath}(\theta_{\tilde m})$ for all $n\in\mathcal J_3$,
it follows from Lemma \ref{lem-x} (iv) that
\be\label{eq4}
\lim_{H\rightarrow+\infty}\frac1H\int_{H}^{2H}\sum_{n\in\mathcal J_3}
A_n(\theta_m)\;e^{i\alpha_n(\theta_m)x_1+i[\beta_n(\theta_m)-\beta_{\imath}(\theta_{\tilde m})]x_2}dx_2=0
\en
uniformly for all $x_1\!\in\!\mathbb{R}$.
This, together with (\ref{eq3})--(\ref{eq4}), implies that (\ref{1119+3}) holds.

Combining (\ref{1119+1})--(\ref{1119+3}), we arrive at
\be\label{1119+4}
\sum_{m\in\mathcal I_{\tilde m}}c_me^{i\alpha_{\imath}(\theta_m)x_1}=0,\quad x_1\in\R.
\en
Multiplying (\ref{1119+4}) by $e^{-i\alpha_{\imath}(\theta_{\tilde m})x_1}$ we obtain
\ben
\sum_{m\in\mathcal K_{\tilde m}}c_m+\sum_{m\in\mathcal I_{\tilde m}\ba\mathcal K_{\tilde m}}c_m e^{i[\alpha_{\imath}(\theta_m)-\alpha_{\imath}(\theta_{\tilde m})]x_1}=0,\quad x_1\in\R,
\enn
where
$\mathcal K_{\tilde m}:=\{m\in\mathcal I_{\tilde m}:\alpha_{\imath}(\theta_m)
=\alpha_{\imath}(\theta_{\tilde m})\}$.
Obviously, $\mathcal K_{\tilde m}\!=\!\{\tilde m\}$.
Then it follows from Lemma \ref{lem-x} (iv) that $c_{\tilde m}\!=\!0$
By the arbitrariness of $\tilde m$ it follows that $c_m\!=\!0$ for all $m=1,\ldots,M$, implying that
$\{u(x;\theta_m)\}_{m=1}^M$ are linearly independent functions in $\Omega$.
\end{proof}

\begin{rem}\label{rem0}{\rm
By (\ref{ui}) the total field $u$ to the diffraction problem (\ref{plane-wave})--(\ref{Rayleigh-up})
is given by
\be\label{211108-7}
u(x)=\sum_{n\in\Z\cup\{\imath\}}A_ne^{i\alpha_n x_1+i\beta_n x_2},\quad x\in U_h.
\en
We claim that
\be\label{eq5}
u \not\equiv 0~\textrm{in} ~\Om.
\en
Assume to the contrary that $u\equiv0$ in $\Om$. Then,
proceeding as in the proof of Lemma \ref{lem-x2}, we first multiply (\ref{211108-7}) by
$e^{-i\beta_{\imath}x_2}$ and then by $e^{-i\alpha_{\imath}x_1}$ to obtain that $A_{\imath}=0$,
which contradicts to the fact that $A_{\imath}=1$.
This implies that (\ref{eq5}) holds.
}
\end{rem}

In the remaining part of this paper, we consider two periodic curves $\G^{(1)}$ and $\G^{(2)}$
with periods $L_1>0$ and $L_2>0$, respectively.
Denote by $\Omega_{j}$ the unbounded connected domain above $\G^{(j)}$ for $j=1,2$.
Set $\G_h:=\{x:x_2=h\}$ for some $h>\max\{x_2\!:\!x\in\G^{(1)}\cup\G^{(2)}\}$.
Denoted by $u^s_j(x;\theta)$ and $u_j(x;\theta)$ the scattered field and total field, respectively,
for incident plane wave $u^i(x;\theta)$ with $\theta\!\in\!(-\pi/2,\pi/2)$ corresponding to
the curve $\G^{(j)}$, $j=1,2$.
Analogously, denote by $(\alpha_n^{(j)},\beta_n^{(j)})$ the pair $(\alpha_n,\beta_n)$ (see (\ref{beta})
and (\ref{ui})) and by $A_n^{(j)}$ the Rayleigh coefficient $A_n$ in (\ref{Rayleigh-up}) and (\ref{ui})
corresponding to $\Gamma=\Gamma^{(j)}$ for $n\in\Z\cup\{\imath\}$ and $j=1,2$.

\section{Determination of grating period from phased data}\label{sec3}
\setcounter{equation}{0}

In this section we consider the inverse problem, that is, whether it is possible to determine
the period of a periodic curve from phased near-field data corresponding to one incident plane wave.
Since the total field $u$ to the forward diffraction model (\ref{plane-wave})--(\ref{Rayleigh-up})
is required to be $\al$-quasi-periodic, it is seen that $e^{-i\alpha x_1}u(x)$ is $L$-periodic
with respect to $x_1$. Actually, this is also implied by (\ref{Rayleigh-up}) and (\ref{ui}).
However, the period $L$ may not be the minimum period of $e^{-i\alpha x_1}u(x)$, as illustrated
in the following remark which presents two diffraction grating curves with different minimum periods
which can generate identical near-field data for one incident plane wave. Such an example was
motivated by the classification of unidentifiable polygonal diffraction gratings using one incident
plane wave; see \cite{B2010,BZZ,EH2010,EH2011}.

\begin{rem}{\rm
Consider the example with $u=u^i+u^s$, where
\be\label{814-1}
u^i(x)=e^{i(-x_1-\sqrt{3}x_2)},\quad u^s(x)=e^{i(x_1+\sqrt{3}x_2)}-e^{-2i x_1}-e^{2i x_1}.
\en
Obviously, $u^i$ is a plane wave defined as in (\ref{plane-wave}) with incident angle
$\theta=-\pi/6$ and wave number $k=2$, implying that $\alpha=-1$.
Note that, if we choose the period $L=2\pi$ then the Rayleigh frequency occurs
(since $\beta_{-1}=\beta_3=0$ in this case). A straightforward calculation shows
\be\label{total-u}
u(x)=2\cos(x_1+\sqrt{3}x_2)-2\cos(2x_1)=-4\sin\frac{3x_1+\sqrt{3}x_2}{2}\,\sin\frac{-x_1+\sqrt{3}x_2}{2}.
\en
Therefore, the zeros of $u(x)$ consist of two families of parallel lines:
\ben
l_n^{(1)}:=\{x=(x_1, x_2)\in \R^2: 3x_1+\sqrt{3}x_2=2n\pi\},\\
l_n^{(2)}:=\{x=(x_1, x_2)\in \R^2: -x_1+\sqrt{3}x_2=2n\pi\}
\enn
for $n\in\Z$, which form a grid in $\R^2$, as illustrated by Figure \ref{counterexample}.
\begin{figure}[hbtp]
  \centering
  \includegraphics[width=0.95\textwidth]{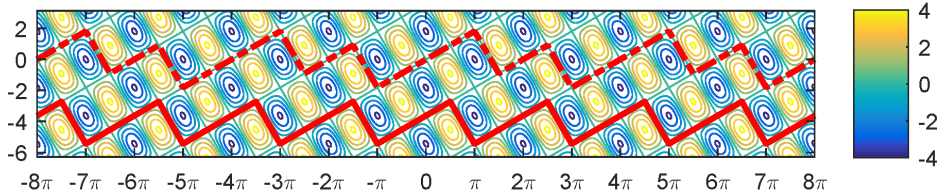}
  \caption{Contour of the total field $u$ given by (\ref{total-u}).
   The red solid line '-' and the red dash-dot line '-.' denote
   two grating curves with different minimum periods.
   }\label{counterexample}
\end{figure}
It is obvious that the two curves $\G^{(1)}$ and $\G^{(2)}$ plotted by the red solid line '-' and
the red dashed line '-.', respectively, as shown in Figure \ref{counterexample}, lie on the above grid.
The minimum period of $\G^{(1)}$ and $\G^{(2)}$ is $L_1=2\pi$ and $L_2=4\pi$, respectively.
From the above discussions and the formula (\ref{814-1}), it can be seen that
$u^s$ is the scattered field to the diffraction problem (\ref{plane-wave})--(\ref{Rayleigh-up})
with the curve $\Gamma\!=\!\Gamma^{(1)}$ and the period $L=L_1$,
and satisfies the Rayleigh expansion (\ref{Rayleigh-up}) with nonzero
Rayleigh coefficients $A_2^{(1)}\!=\!1$, $A_{-1}^{(1)}\!=\!A_3^{(1)}\!=\!-1$.
However, on the other hand, it is also easily seen that $u^s$ is the scattered field
to the diffraction problem (\ref{plane-wave})--(\ref{Rayleigh-up})
with the curve $\Gamma\!=\!\Gamma^{(2)}$ and the period $L\!=\!L_2$,
and satisfies the Rayleigh expansion (\ref{Rayleigh-up}) with nonzero Rayleigh
coefficients $A_4^{(2)}\!=\!1$, $A_{-2}^{(2)}\!=\!A_6^{(2)}\!=\!-1$.
This example shows that it is impossible to determine the minimum period (also the shape)
of a grating curve from phased near-field data corresponding to one incident plane wave.
}
\end{rem}

In general, one can only find a common period of two grating curves if their scattered fields coincide.
This will be proved rigorously in Theorem \ref{TH-phase-period} below, where the
periodic curves do not need to satisfy the smoothness Conditions (i) and (ii).

\begin{thm}\label{TH-phase-period}
Suppose $\theta\in(-\pi/2,\pi/2)$ is an arbitrarily fixed incident angle.
Let $\Gamma^{(1)}$ and $\Gamma^{(2)}$ be two periodic curves.
If the corresponding scattered fields satisfy
\be\label{x-1}
u^s_1(x;\theta)\!=\!u^s_2(x;\theta)\quad
\mbox{on}\quad x_2\!=\!h>\max\{x_2\!:\!x\!\in\!\G^{(1)}\cup\G^{(2)}\},
\en
then there exists $L>0$ such that $L$ is a period of both $\Gamma^{(1)}$ and $\Gamma^{(2)}$.
\end{thm}

\begin{proof}
Suppose $L_j>0$ is a period of the curve $\G^{(j)}$, $j\!=\!1,2$. Then the corresponding
scattered field $u^s_j(x;\theta)$ satisfies the following Rayleigh expansions
\be\label{z-1}
u^s_j(x)=\sum_{n\in\Z}A_n^{(j)}e^{i\alpha^{(j)}_nx_1+i\beta^{(j)}_nx_2},\quad
x\in U_h:=\{x\in\R^2:x_2>h\},j=1,2,
\en
where $\alpha^{(j)}_n$, $\beta_n^{(j)}$ and the coefficients $A_n^{(j)}$, that depends on $k$, $\theta$
and $\Gamma^{(j)}$, are defined analogously to $\alpha_n$, $\beta_n$ and $A_n$ with $\Gamma$
replaced by $\Gamma_j$. Note that the following conditions are fulfilled:

(i) $u^s_1-u^s_2$ satisfies the Helmholtz equation in $U_h$;

(ii) $u^s_1-u^s_2=0$ on $\G_h:=\{x:x_2=h\}$;

(iii) $\sup_{x\in U_h}|u^s_1(x)-u^s_2(x)|<{\color{hw}+\infty}$;

(iv) $u^s_1-u^s_2$ satisfies the upward propagating radiation condition
(see \cite[Definition 2.2]{Zhang98}).\\
In fact, (i) follows from (\ref{plane-wave}) and (\ref{equation}), and (ii) follows from (\ref{x-1}).
(iii) and (iv) are implied by the Rayleigh expansions (\ref{z-1}) (see Lemma \ref{lem-x} (i)
and \cite[pp. 1777]{Zhang98}).
By uniqueness to the Dirichlet boundary value problem in $U_h$ (see \cite[Theorem 3.4]{Zhang98}),
it follows that
\be\label{x-eq2}
u^s_1(x;\theta)=u^s_2(x;\theta),\quad x\in U_h.
\en

We now consider the following two cases.

\textbf{Case 1:} ${L_1}/{L_2}$ is rational.

Let $p/q={L_1}/{L_2}$ with reduced fraction $p/q$ and positive integers $p,q\in\Z_+$.
Set $L:=qL_1$. Then $L=pL_2$.
Thus $L$ is a common period for both $\G^{(1)}$ and $\G^{(2)}$.

\textbf{Case 2:} ${L_1}/{L_2}$ is irrational.

We claim that any $L\!>\!0$ is a period of both $\G^{(1)}$ and $\G^{(2)}$.
To do this, we first deduce from the fact that ${L_1}/{L_2}$ is irrational that
\be\label{neq}
\alpha^{(1)}_m\!\neq\!\alpha^{(2)}_n\;\text{for all}\;(m,n)\!\in\!\Z^2\ba\{(0,0)\}\text{ and }\alpha^{(1)}_0\!=\!\alpha^{(2)}_0\!=\!k\sin\theta.
\en
It follows from (\ref{z-1}) and (\ref{x-eq2}) that
\be\label{814-2}
\sum_{n\in\Z}A_n^{(1)}e^{i\alpha^{(1)}_nx_1+i\beta^{(1)}_nx_2}-\sum_{n\in\Z}
A_n^{(2)}e^{i\alpha^{(2)}_nx_1+i\beta^{(2)}_nx_2}=0,\quad x\in U_h.
\en
The proof of this case can be divided into three steps as follows.

\textbf{Step 1.} We prove that
\be\label{1120-1}
&&A_0^{(1)}\!=\!A_0^{(2)},\\ \label{1120-1-}
&&A_n^{(1)}\!=\!0\text{ for all }n\in\Z\ba\{0\}\text{ such that }|\alpha_n^{(1)}|\leq k,\\ \label{1120+1}
&&A_n^{(2)}\!=\!0\text{ for all }n\in\Z\ba\{0\}\text{ such that }|\alpha_n^{(2)}|\leq k.
\en
Let $\tilde n\!\in\!\Z$ be arbitrarily fixed such that $|\alpha_{\tilde n}^{(1)}|\!\leq\!k$.
Multiplying (\ref{814-2}) by $e^{-i\beta^{(1)}_{\tilde n}x_2}$ we obtain
\be\no
&&\sum_{n\in\mathcal I^{(1)}_{\tilde n}}A_n^{(1)}e^{i\alpha^{(1)}_nx_1}
+\sum_{n\in\Z\ba\mathcal I^{(1)}_{\tilde n}}A_n^{(1)}
e^{i\alpha^{(1)}_nx_1+i(\beta^{(1)}_n-\beta^{(1)}_{\tilde n})x_2}\\ \label{211107-3}
&&-\sum_{n\in\mathcal I^{(2)}_{\tilde n}}A_n^{(2)}e^{i\alpha^{(2)}_nx_1}
-\sum_{n\in\Z\ba\mathcal I^{(2)}_{\tilde n}}A_n^{(2)}
e^{i\alpha^{(2)}_nx_1+i(\beta^{(2)}_n-\beta^{(1)}_{\tilde n})x_2}=0,\quad x\in U_h,
\en
where $\mathcal I^{(j)}_{\tilde n}:=\{n\in\Z:\beta^{(j)}_n=\beta^{(1)}_{\tilde n}\}$
is at most a finite set for $j=1,2$.
Analogously to (\ref{1119+3}), using $|e^{i\beta^{(1)}_{\tilde n}x_2}|=1$, we can apply
Lemma \ref{lem-x} (ii) and (iv) to obtain
\ben
\lim_{H\rightarrow{\color{hw}+\infty}}\frac1H\int_H^{2H}\sum_{n\in\Z\ba\mathcal I^{(j)}_{\tilde n}}
A_n^{(j)}e^{i\alpha^{(j)}_nx_1+i(\beta^{(j)}_n-\beta^{(1)}_{\tilde n})x_2}dx_2=0,\quad j=1,2,
\enn
for all $x_1\!\in\!\mathbb{R}$.
Therefore, it follows from (\ref{211107-3}) that
\be\label{211107-4}
\sum_{n\in\mathcal I^{(1)}_{\tilde n}}A_n^{(1)}e^{i\alpha^{(1)}_nx_1}
-\sum_{n\in\mathcal I^{(2)}_{\tilde n}}A_n^{(2)}e^{i\alpha^{(2)}_nx_1}=0,\quad x_1\in\R.
\en
Similarly, multiplying (\ref{211107-4}) by $e^{-i\alpha^{(1)}_{\tilde n}x_1}$ we can deduce
from Lemma \ref{lem-x} (iv) that
\be\label{211108-}
\sum_{n\in\mathcal K^{(1)}_{\tilde n}}A_n^{(1)}-\sum_{n\in\mathcal K^{(2)}_{\tilde n}}A_n^{(2)}=0.
\en
where $\mathcal K^{(j)}_{\tilde n}:=\{n\in\Z:\alpha^{(j)}_{n}=\alpha^{(1)}_{\tilde n},
\beta^{(j)}_n=\beta^{(1)}_{\tilde n}\}$, $j=1,2$.
Obviously, $\mathcal K^{(1)}_{\tilde n}=\{\tilde n\}$.
In view of (\ref{neq}), we know that $\mathcal K^{(2)}_{\tilde n}=\{0\}$ if $\tilde n=0$ and
$\mathcal K^{(2)}_{\tilde n}=\emptyset$ if $\tilde n\in\Z\ba\{0\}$.
These, together with (\ref{211108-}), imply (\ref{1120-1}) and (\ref{1120-1-}).
By interchanging the role of $u_1^s$ and $u_2^s$, {\color{hw}we can employ} a similar argument
as above to obtain (\ref{1120+1}).

\textbf{Step 2.} We prove that
\be\label{1120-2}
&&A_n^{(1)}\!=\!0\text{ for all }n\in\Z\text{ such that }|\alpha_n^{(1)}|>k,\\ \label{1120-2-}
&&A_n^{(2)}\!=\!0\text{ for all }n\in\Z\text{ such that }|\alpha_n^{(2)}|>k.
\en
Set $\mathcal P^{(j)}\!:=\!\{n\!\in\!\Z\!:\!|\alpha^{(j)}_n|\!>\!k\}$, $j\!=\!1,2$.
It follows from (\ref{814-2})--(\ref{1120+1}) that
\be\label{211107-5}
\sum_{n\in\mathcal P^{(1)}}A_n^{(1)}e^{i\alpha^{(1)}_nx_1+i\beta^{(1)}_nx_2}-\sum_{n\in\mathcal P^{(2)}}A_n^{(2)}e^{i\alpha^{(2)}_nx_1+i\beta^{(2)}_nx_2}=0,\quad x\in U_h.
\en
By (\ref{beta}), we can rearrange the elements in
$\{(1,n):n\in\mathcal P^{(1)}\}\cup\{(2,n):n\in\mathcal P^{(2)}\}$ as a sequence
$\{(p_\ell,q_\ell)\}_{\ell\in\Z_+}$ such that $\beta^{(p_\ell)}_{q_\ell}=ib_\ell$ with $b_\ell>0$ and
$b_\ell\le b_{\ell+1}$ for all $\ell\in\Z_+$.
Obviously, $b_\ell\!\rightarrow\!+\infty$ as $\ell\rightarrow+\infty$.

Without loss of generality, we may assume that $p_1=1$ and $q_1=\tilde n$ for some
${\tilde n}\in\mathcal P^{(1)}$ and thus $\beta^{(p_1)}_{q_1}=\beta^{(1)}_{\tilde n}$.
Let $\mathcal I^{(j)}_{\tilde n}$ ($j\!=\!1,2$) be defined as in Step 1.
It is clear that $\mathcal I^{(j)}_{\tilde n}=\{n\in\mathcal P^{(1)}:\beta^{(j)}_n=\beta^{(1)}_{\tilde n}\}$
and is at most a finite set.
Then, multiplying (\ref{211107-5}) by $e^{-i\beta^{(1)}_{\tilde n}x_2}$ we obtain
\be\no
&&\sum_{n\in\mathcal I^{(1)}_{\tilde n}}A_n^{(1)}e^{i\alpha^{(1)}_nx_1}+\sum_{n\in\mathcal P^{(1)}\ba\mathcal I^{(1)}_{\tilde n}}A_n^{(1)}e^{i\alpha^{(1)}_nx_1+i(\beta^{(1)}_n-\beta^{(1)}_{\tilde n})x_2}\\ \label{211108-3}
&&-\sum_{n\in\mathcal I^{(2)}_{\tilde n}}A_n^{(2)}e^{i\alpha^{(2)}_nx_1}-\sum_{n\in\mathcal P^{(2)}\ba\mathcal I^{(2)}_{\tilde n}}A_n^{(2)}e^{i\alpha^{(2)}_nx_1+i(\beta^{(2)}_n-\beta^{(1)}_{\tilde n})x_2}=0,\;x\in U_h.
\en

For $N>0$ large enough and $j=1,2$, we set
$\mathcal Q^{(j)}_1(N):=\{n\in\mathcal P^{(j)}\ba\mathcal I^{(j)}_{\tilde n}:|n|>N\}$ and
$\mathcal Q^{(j)}_2(N):=\{n\in\mathcal P^{(j)}\ba\mathcal I^{(j)}_{\tilde n}:|n|\le N\}$.
By Lemma \ref{lem-x} (iii), we have
\be\label{eq6}
\lim_{N\rightarrow+\infty}\sum_{n\in\mathcal Q^{(j)}_1(N)}\left|A_n^{(j)}e^{i\alpha^{(j)}_nx_1+i(\beta^{(j)}_n-\beta^{(1)}_{\tilde n})x_2}\right|=0,
\quad j=1,2,
\en
uniformly for all $x\!\in\!U_h$.
For any fixed $N\!>\!0$, since $\mathcal Q_2^{(j)}(N)$ is a finite set and
$i(\beta^{(j)}_n\!-\!\beta^{(1)}_{\tilde n})\!<\!0$ for all $n\!\in\!\mathcal Q_2^{(j)}(N)$
due to the definition of $\beta^{(1)}_{\tilde n}$, thus we have
\be\label{eq7}
\lim_{x_2\rightarrow+\infty}\sum_{n\in\mathcal Q_2^{(j)}(N)}\left|A_n^{(j)}e^{i\alpha^{(j)}_nx_1+i(\beta^{(j)}_n-\beta^{(1)}_{\tilde n})x_2}\right|=0,
\quad j=1,2,
\en
uniformly for all $x_1\!\in\!\mathbb{R}$.
Thus, it follows from (\ref{eq6}) and (\ref{eq7}) that
\ben%\label{1120-3}
\lim_{x_2\rightarrow+\infty}\sum_{n\in\mathcal P^{(j)}\ba\mathcal I^{(j)}_{\tilde n}}
A_n^{(j)}e^{i\alpha^{(j)}_nx_1+i(\beta^{(j)}_n-\beta^{(1)}_{\tilde n})x_2}=0,\quad j=1,2,
\enn
for all $x_1\!\in\!\mathbb{R}$.
This, together with (\ref{211108-3}), implies that (\ref{211107-4}) holds.
Analogously to Step 1, multiplying (\ref{211107-4}) by $e^{-i\alpha^{(1)}_{\tilde n}x_1}$,
we can apply Lemma \ref{lem-x} (iv) to obtain (\ref{211108-}) and thus
$A^{(1)}_{\tilde n}\!=\!{\color{hw}A^{(p_1)}_{q_1}}\!=\!0$.
Taking this into (\ref{211107-5}), we obtain that (\ref{211107-5}) holds with $\mathcal P^{(1)}$
replaced by $\mathcal P^{(1)}\backslash\{q_1\}$. Then using the same argument as above, we can
obtain that $A^{(p_2)}_{q_2}=0$. Now, we can repeat the same
argument again to obtain that $A^{(p_\ell)}_{q_\ell}=0$ for all $\ell\in \mathbb{Z}_+$.
This means that (\ref{1120-2}) and (\ref{1120-2-}) hold.

\textbf{Step 3.}
Combining (\ref{1120-1})--(\ref{1120+1}), (\ref{1120-2}) and (\ref{1120-2-}), we arrive at
\ben
A_0^{(1)}=A_0^{(2)}\;\text{and}\;A_n^{(1)}=A_n^{(2)}=0\;\text{for}\;n\in\Z\ba\{0\}.
\enn
Then by the Dirichlet boundary condition imposed on $\G^{(j)}$ ($j=1,2$), we have
\ben
e^{i\alpha_0x_1-i\beta_0x_2}=u^i(x)=-u^s_j(x)=-A_0^{(j)}e^{i\alpha_0x_1+i\beta_0x_2},\quad x\in\Gamma^{(j)},j=1,2.
\enn
This further implies that $\G^{(j)}$ ($j=1,2$) is a straight line parallel to the $x_1$-axis
since $A_0^{(j)}$ is a constant.
Thus, any $L\!>\!0$ is a common period of $\G^{(1)}$ and $\G^{(2)}$.
\end{proof}

\section{Uniqueness with phased data}\label{sec4}
\setcounter{equation}{0}

In this section, we prove that a periodic curve with Dirichlet boundary condition fulfilling
Condition (i) or Condition (ii) can be uniquely determined by the fixed-frequency near-field data
corresponding to incident plane waves with distinct angles (i.e., Theorem \ref{TH-main} with phased data).
This differs from \cite{HF97}, where fixed-direction incident plane waves with different frequencies
are used, and this also differs from \cite{K94} which involves fixed-frequency quasi-periodic incident
waves with the same phase shift.
For the inverse problem to recover a periodic curve from near-field data corresponding
to incident plane waves with distinct directions, difficulties arise from the fact that
the corresponding total fields have different phase shifts since $\alpha=k\sin\theta$ depends
on the incident angle $\theta$.
We rephrase Theorem \ref{TH-main} with phased data in Theorem \ref{TH-phase} below, which is the
main uniqueness result of this section. Here we shall provide a proof based on both the ideas of Schiffer
for bounded obstacles (see \cite{CS1983}) and for periodic structures with multi-frequency data
(see \cite{HF97}) and the concept of dispersion relations (see, e.g., \cite{FJ16, Kuchment,SW2002})
arising from the analysis of photonic crystals.

\begin{thm}\label{TH-phase}
Let $\G^{(1)}$ and $\G^{(2)}$ be two periodic curves with Dirichlet boundary conditions.
Assume both of them satisfy Condition (i) or both of them satisfy Condition (ii).
Suppose that the periods of $\G^{(1)}$ and $\G^{(2)}$ are unknown.
If the corresponding total fields satisfy
\be\label{x-1-}
u_1(x;\theta_n)=u_2(x;\theta_n),\quad x\in \mathcal{S},\;n\in\Z_+,
\en
where $\{\theta_n\}_{n=1}^\infty$ are distinct incident angles in $(-\pi/2,\pi/2)$, then $\G^{(1)}=\G^{(2)}$.
Here, $\mathcal S\!\subset\!\G_h$ is a line segment with $\G_h\!:=\!\{x\!:\!x_2\!=\!h\}$ and $h\!>\!\max\{x_2\!:\!x\!\in\!\G^{(1)}\!\cup\!\G^{(2)}\}$ being an arbitrary constant.
\end{thm}

Since $u_1$ and $u_2$ are analytic functions of $x\!\in\!\G_h$, (\ref{x-1-}) is equivalent to $u_1(x;\theta_n)\!=\!u_2(x;\theta_n)$ for all $x\!\in\!\G_h$ and $n\!\in\!\Z_+$.
Therefore, $u^s_1(x;\theta_n)\!=\!u^s_2(x;\theta_n)$ for all $x\!\in\!\G_h$ and $n\!\in\!\Z_+$.
Analogously to (\ref{x-eq2}), we have $u^s_1(x;\theta_n)\!=\!u^s_2(x;\theta_n)$ for all $x\in U_h$
and $n\in\Z_+$.
By analyticity we arrive at
\be\label{eq2}
u^s_1(x;\theta_n)=u^s_2(x;\theta_n),\quad x\in\Om',n\in \Z_+,
\en
where $\Omega'$ denotes the unbounded component of $\Omega_{1}\!\cap\!\Omega_{2}$ which can be
connected to $U_h$.
By Theorem \ref{TH-phase-period}, the above relation $(\ref{eq2})$ implies that there exists $L\!>\!0$
such that $L$ is a common period of $\Gamma^{(1)}$ and $\Gamma^{(2)}$.
Without loss of generality, we may assume $L\!=\!2\pi$ in the rest of this section.
Assume to the contrary that $\Gamma^{(1)}\!\neq\!\Gamma^{(2)}$.
We need to consider the following two cases:
\ben
\textbf{Case (i)}: \Gamma^{(1)}\cap\Gamma^{(2)}\neq\emptyset;\quad\quad
\textbf{Case (ii)}: \Gamma^{(1)}\cap\Gamma^{(2)}=\emptyset.
\enn
The proofs of Theorem \ref{TH-phase} for these two cases will be given in the following subsections.

\subsection{Proof of Theorem \ref{TH-phase} for \mbox{Case (i):} $\Gamma^{(1)}\cap \Gamma^{(2)}\neq\emptyset$.}\label{sub-case1}

\begin{figure}[!ht]
  \centering
  \subfigure[Both $\G^{(1)}$ and $\G^{(2)}$ are graphs]{
  \includegraphics[width=0.45\textwidth]{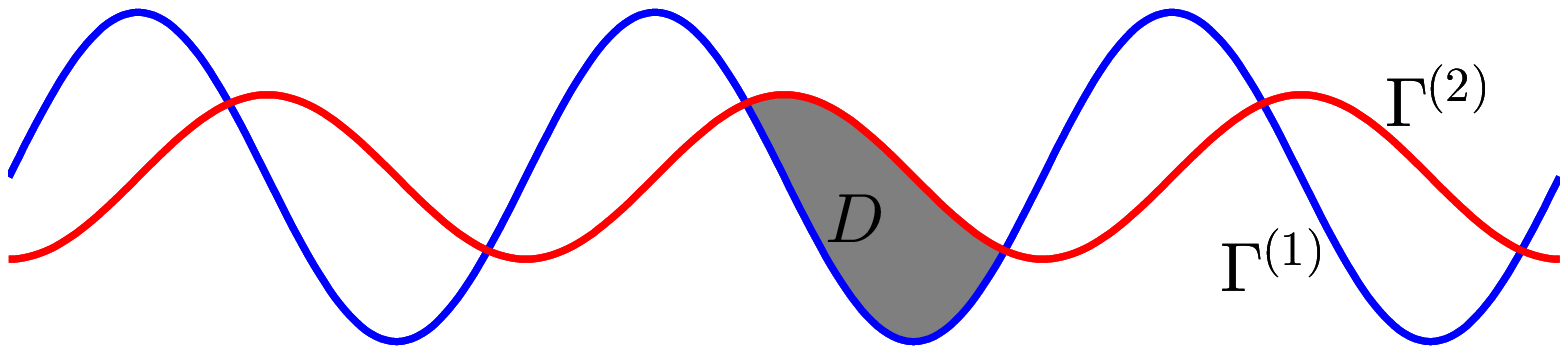}}
  \subfigure[Neither $\G^{(1)}$ nor $\G^{(2)}$ is a graph]{
  \includegraphics[width=0.45\textwidth]{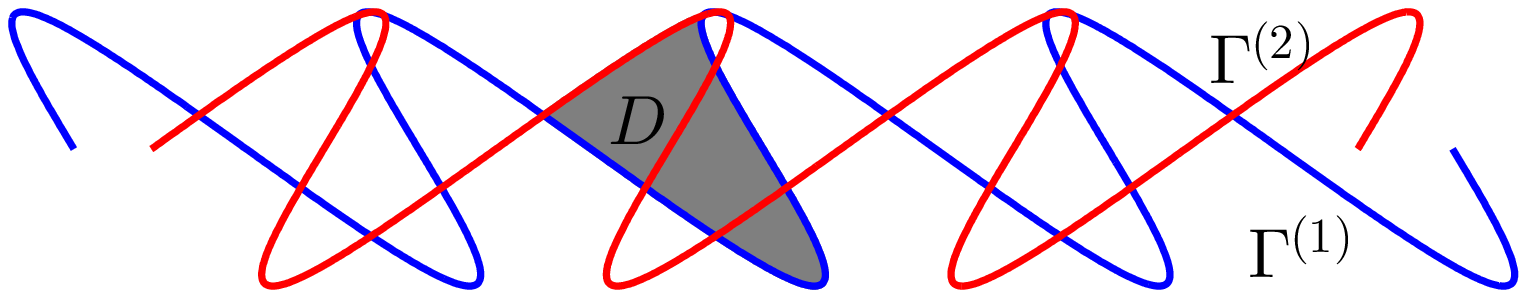}}
  \caption{The bounded domain $D$ in Case (i): $\Gamma^{(1)}\cap\Gamma^{(2)}\neq\emptyset$.}
  \label{fig-case1}
\end{figure}
Since $\Gamma^{(1)}\!\cap\!\Gamma^{(2)}\!\neq\!\emptyset$ and both $\G^{(1)}$ and $\G^{(2)}$
are $2\pi$-periodic, there exists at least one bounded domain $D$ enclosed by $\Gamma^{(1)}$ and $\Gamma^{(2)}$.
In other words, $\partial D\!\subset\!\Gamma^{(1)}\!\cup\!\Gamma^{(2)}$.
Without loss of general we may suppose that $D\subset\Om_{1}\ba\ov{\Om'}$
as shown in Figure \ref{fig-case1}.
It follows from Remark \ref{rem0}, formula (\ref{eq2}) and the Dirichlet boundary condition
of $u_j(x;\theta_n)$ on $\G^{(j)}$ that the total field $u_1(x;\theta_n):=u^i(x;\theta_n)+u^s_1(x;\theta_n)$
is a nontrivial solution to the eigenvalue problem
\ben
\Delta u+k^2u=0\quad\mbox{in }D,\quad u=0\quad\mbox{on }\partial D,
\enn
for all $n\!\in\!\Z_+$.
In other words, $u_1(x;\theta_n)$ is a Dirichlet eigenfunction of the negative Laplacian
in $D$ for each $n\!\in\!\Z_+$.
Recall from Lemma \ref{lem-x2} that $\{u_1(x;\theta_n)\}_{n=1}^N$ are linearly independent
functions in $D$ for any positive integer $N<+\infty$.
However, by a similar argument as in the proof of \cite[Theorem 5.1]{CK19},
it follows that there are at most finitely many independent Dirichlet eigenfunctions of the negative
Laplacian in $H_0^1(D)$ corresponding to the eigenvalue $k^2\!>\!0$.
This contradiction implies that Case (i) does not hold.

\begin{rem}{\rm
It should be remarked that, the proof of \cite[Theorem 5.1]{CK19} relies essentially on the a priori
estimate of solutions after the Gram-Schmidt orthogonalization of $\{u_1(x;\theta_n)\}_{n\in\Z_+}$
(see \cite[the third formula on page 140]{CK19}).
However, if $D$ is an unbounded periodic strip, as will be seen in Case (ii), it would be difficult
to establish an analogous a priori estimate of solutions with different incident angles (or equivalently,
with different phase shifts $k\sin\theta_n$) after the Gram-Schmidt orthogonalization. Hence,
the aforementioned arguments cannot be used for treating Case (ii).
}
\end{rem}

\subsection{Proof of Theorem \ref{TH-phase} for \mbox{Case (ii):} $\Gamma^{(1)}\cap \Gamma^{(2)}=\emptyset$.}\label{sub-case2}

We suppose without loss of generality that $\Gamma^{(2)}$ lies entirely above $\Gamma^{(1)}$ as shown
in Figure \ref{case2}.
\begin{figure}[!ht]
  \centering
  \subfigure[Both $\G^{(1)}$ and $\G^{(2)}$ are graphs]{
  \includegraphics[width=0.45\textwidth]{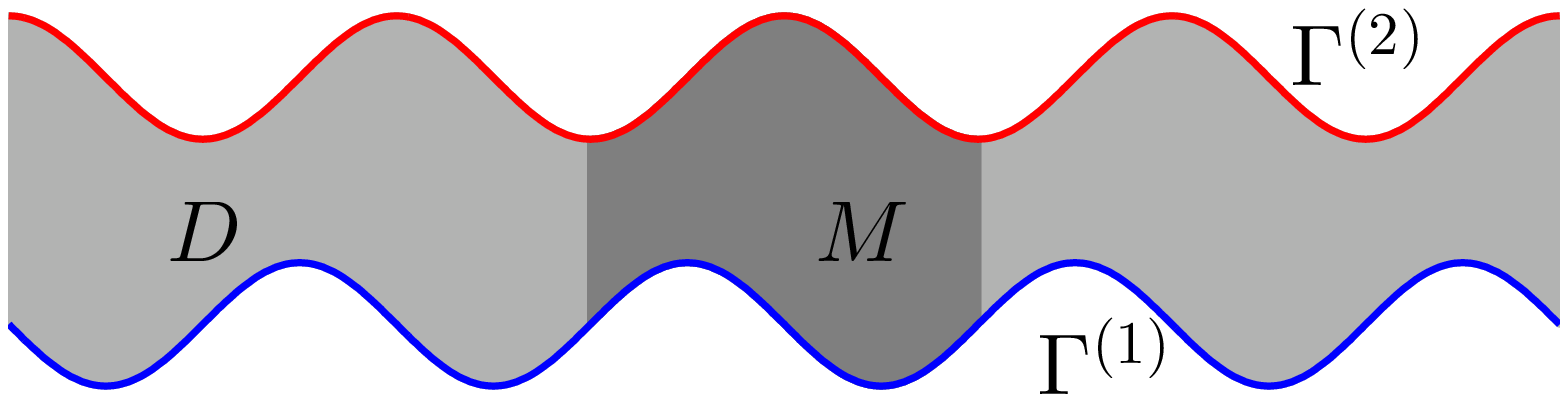}}
  \subfigure[$\G^{(1)}$ is not a graph and $\G^{(2)}$ is a graph]{
  \includegraphics[width=0.45\textwidth]{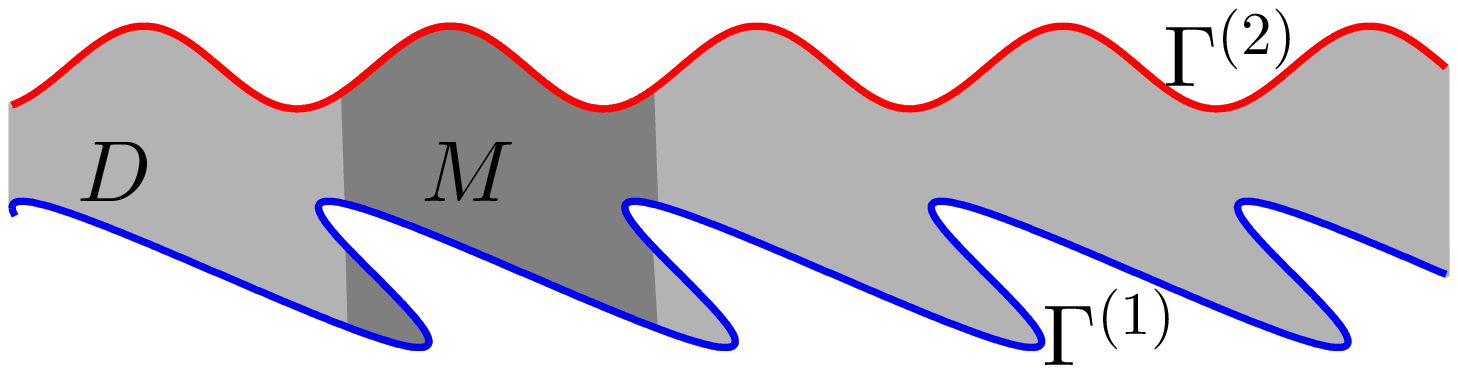}}
  \caption{The unbounded periodic strip $D$ and its one periodic cell $M$ in Case (ii): $\G^{(1)}\cap \G^{(2)}=\emptyset$.}\label{case2}
\end{figure}
Denote by $D$ the unbounded $2\pi$-periodic strip (waveguide) lying between the two curves.
To investigate the dependance of solutions on the quasi-periodic shift $\alpha=\alpha(\theta_n)=k\sin\theta_n$,
we set $w_n(x)\!:=\!e^{-i\alpha(\theta_n) x_1} u_1(x;\theta_n)$.
It then follows from (\ref{equation}) and (\ref{eq2}) that  $w_n$ satisfies
the periodic boundary value problem
\be\label{PBVP}
\left\{\begin{array}{lll}
 \nabla_{\alpha(\theta_n)} \cdot \nabla_{\alpha(\theta_n)}w_n+k^2w_n=0\quad\mbox{in}\quad D,\\
w_n=0\quad\mbox{on}\quad \Gamma^{(1)}\cup \Gamma^{(2)},\\
\mbox{$w_n$ is $2\pi$-periodic with respect to $x_1$ in $D$,}
\end{array}\right.
\en for all $n\in\Z_+$, where $\nabla_{\alpha(\theta_n)}:=(\partial_1+i\alpha(\theta_n),\partial_2)^\top$.
For $\alpha=k\mu$ with $\mu\in(-1,1)$, we consider the abstract Dirichlet boundary value problem
in a closed periodic waveguide $D$:
\ben {\rm (BVP)}\quad
\left\{\begin{array}{lll}
\nabla_\alpha\cdot\nabla_\alpha w+k^2w=0\quad\mbox{in}\quad D,\\
w=0\quad\mbox{on}\quad \Gamma^{(1)}\cup \G^{(2)},\\
\mbox{$w$ is $2\pi$-periodic with respect to $x_1$ in $D$.}
\end{array}\right.
\enn

\begin{defn}
For any fixed $k\!>\!0$, we say that $\mu\!\in\!(-1,1)$ is called a $\mu$-eigenvalue if the above
boundary value problem (BVP) admits a nontrivial solution in the space
$H^1_{0,0}(D):=\{w\in H^1_{loc}(D):w\;\text{is}\;2\pi\text{-periodic with respect to}~x_1,~w
=0~\textrm{on}~\partial D\}$.
Accordingly, the nontrivial solution is the associated eigenfunction.
\end{defn}

Since $u_1(x;\theta_n)\!\not\equiv\!0$ for $x\!\in\!\Om_{1}$, we conclude from (\ref{PBVP})
that $\sin\theta_n$ is a $\mu$-eigenvalue to (BVP) with the eigenfunction $w_n$ for all $n\!\in\!\Z_+$.
On the other hand, for any fixed $\mu\!\in\!(-1,1)$, we say that $k\!>\!0$ is called a $k$-eigenvalue
if (BVP) admits a nontrivial solution $w\!\in\!H^1_{0,0}(D)$.
As shown in \cite[Theorem 2.3]{HF97}, the $k$-eigenvalues form a discrete set on the positive real-axis
with the only accumulating  point at infinity  and the associated eigenspace for each $k$-eigenvalue
is of finite dimensions.
It is easy to observe that, if $w(x)$ solves (BVP) with $\mu\in(-1,1)$ and some $k_j(\mu)$,
then the conjugate $\overline{w}$ is also a nontrivial solution corresponding to $-\mu$. This implies
the even symmetry of $k_j(\mu)$ with respect to the line $\mu=0$, that is, $k_j(\mu)=k_j(-\mu)$
for each $\mu\in (-1,1)$.

The $\alpha$-dependent partial differential equation in (BVP) can be regarded as the Floquet-Bloch (FB)
transform of the Helmholtz equation $(\Delta+k^2)u=0$ in the $x_1$-direction with the
variable $\alpha\in\R$; see \cite{Kuchment, FJ16}. The Bloch theory in one direction was well-summarized
in \cite[Section 3]{FJ16} for deriving physically-meaningful radiation conditions in a closed periodic
waveguide.

Let us now recall the dispersion relations for the $2\pi$-periodic system (BVP), where the FB transform
variable $\alpha\in \R$ is independent of $k$. For each $\alpha\in\R$, there also exists a discrete set
of numbers $K_j(\alpha)>0$ such that the boundary value problem (BVP) admits non-trivial solutions
with $k^2=K_j(\alpha)$ for each $j=1,2,\ldots$ (see Remark \ref{rem0920} below).
By \cite[Chapter 7]{Kato}, the function $\alpha\rightarrow K_j(\alpha)$ is continuous and piecewise
analytic. Further, $K_j(\alpha)$ is not analytic at $\alpha=\alpha_0$ only if $k^2=K_j(\alpha_0)$
is not a simple eigenvalue.
Recall from \eqref{qp} with $L=2\pi$ that an $\alpha$-quasiperiodic function must also
be $(\alpha+j)$-quasiperiodic for any $j\in\N$.
It is easy to conclude that $K_j(\alpha):\R\rightarrow\R$ is periodic in $\alpha$ with
the periodicity one. Restricting to one periodic interval $[-1/2,1/2]$, we also have the even symmetry
$K_j(\alpha)=K_j(-\alpha)$ for all $\alpha\in[-1/2, 1/2]$.
The $\alpha$-dependent eigenvalues $K_j(\alpha)$ can be relabelled for $j\in\Z_+$ so as to make
the eigenvalues and associated eigenfunctions analytic in $\alpha\in\R$ (see, e.g.,
\cite[Theorem 3.9, Chapter 7]{Kato} or \cite[Section 3.3]{FJ16}).
For $j\in\Z_+$ the curves given by $K_j(\alpha):(-1/2,1/2]\rightarrow\R$ for the
relabelled indices are well known as dispersion relations, and the graphs of the dispersion
relations define the Bloch variety \cite{Kuchment}. Note that the dispersion curves are no
longer periodic. Below we characterize the relation between the function $\mu\mapsto k(\mu)$
and the dispersion relation $\alpha\mapsto K(\alpha)$.

\begin{lem}\label{lem4.1}
(i) The function $k_j(\mu):(-1,1)\rightarrow\R_+$ must fulfill the dispersion relation
$K_{j'}(\mu\, k_j(\mu))=k_j^2 (\mu)$ for some $j'\in \Z_+$.
Conversely, from the dispersion relation $K_{j'}(\mu k)=k^2$ one can always deduce the function
$k=k_j(\mu)$ for some $j\in \Z_+$.

(ii) If $k_j(\mu)\equiv {\rm Const}$ for some $j\in \Z_+$, then $K_{j'}(\alpha)\equiv{\rm Const}$
for some $j'\in \Z_+$ and vice versa.
\end{lem}

\begin{proof}
(i) The first part follows straightforwardly from the definitions of $k_j$ and $K_{j'}$. To prove
the second part, we set $F(k):=K(\mu k)-k^2$. Obviously, $\D F/ \D k=\mu K'(\mu k)-2 k$,
where $F'(\alpha):=\D F/\D \alpha$. If
\be\label{K}
K(\mu k)-k^2=0,\quad \mu K'(\mu k)-2k=0,
\en we can conclude that
\ben
\alpha\,K'(\alpha)-2K(\alpha)=0,\quad \alpha=\mu k.
\enn
Hence, $K(\alpha)=c\, \alpha^2$ for some constant $c\in \R$. By the 1-periodicity of $K$ we obtain $c=0$
and thus $K\equiv 0$. This further leads to $k=0$ and by integration by part, any solution to (BVP) must
vanish identically. Hence, the two relations in \eqref{K} cannot hold simultaneously. By the implicit
function theorem one can ways get the function $k=k_j(\mu)$ for some $j\in \Z_+$ from the dispersion
relation $K_{j'}(\alpha)=k^2$.

(ii) The second assertion is a direct consequence of the first assertion.
\end{proof}

\begin{rem}{\rm
We consider a special case when $D=\R\times(0,h)$ is a straight strip with some $h>0$.
By separation of variables, it was proved in  \cite{HF97} that the dispersion relation is given by
\be\label{dispersion}
K_{n,m}(\alpha) =(\alpha+n)^2+\left(\frac{m\pi}h\right)^2,\quad n\in\Z,m\in\Z_+,
\en
when $|\alpha|\!<\!k$ (see \cite[(3.5)]{HF97}).
By a same argument as in \cite{HF97}, (\ref{dispersion}) holds for all $\alpha\!\in\!\mathbb{R}$.
Here, the dispersion relation $\{K_{n,m}(\alpha)\}_{n\in\Z,m\in\Z_+}$ is the rearrangement of
$\{K_j(\alpha)\}_{j\in\Z_+}$ mentioned above.
}
\end{rem}

For a proof of Theorem \ref{TH-phase} in Case (ii), it suffices to prove that the $\mu$-eigenvalues
must be discrete for any fixed $k>0$. To this end, we need the following proposition.

\begin{prop}\label{lem-x2.3}
Suppose that $\Gamma^{(1)}$ and $\Gamma^{(2)}$ are both analytic curves or the graphs of $3$-times
continuously differentiable functions such that $\Gamma^{(1)}\cap \Gamma^{(2)}=\emptyset$.
Then the problem (BVP) has no flat dispersion curves, that is, $K_j(\alpha)\not\equiv\, {\rm const}$
for any $j\in\Z_+$.
\end{prop}

The result of Proposition \ref{lem-x2.3} was essentially contained in the proof of
\cite[Theorem 2.3]{SW2002} for general periodic partial differential equations in an open or
closed waveguide. In a closed waveguide, both the Dirichlet and Neumann boundary conditions were
considered there. Moreover, Proposition \ref{lem-x2.3} applies to general 3-admissible periodic
domains (see \cite[Definition 2.2]{SW2002}) which can be obtained from a straight strip by a
periodic $W^{3,\infty}$-mapping/a $3$-admissible mapping, including the periodic strips stated
in Proposition \ref{lem-x2.3}.
As a direct consequence of Proposition \ref{lem-x2.3}, we have the following result.

\begin{cor}\label{lin-indep:cor1}
Let $k>0$ be an arbitrarily fixed wave number.
Under the conditions of Proposition \ref{lem-x2.3}, there exists at least one parameter $\mu\in(-1,1)$
such that the periodic boundary value problem (BVP) admits the trivial solution only.
\end{cor}

\begin{proof}
Assume to the contrary that, for some $k>0$ the periodic boundary value problem (BVP) admits
nontrivial solutions for each $\mu\in(-1, 1)$.
This implies that $k_j(\mu)=k>0$ for all $\mu\in(-1,1)$ and for some $j\in\Z_+$.
By Lemma \ref{lem4.1} (ii), there exists one flat dispersion curve $K_{j'}(\alpha)\equiv k^2$ for
some $j'\in\Z_+$ for the system (BVP),
which contradicts Proposition \ref{lem-x2.3}.
\end{proof}

If $\alpha\in\C$ and $\rm{Im}\,\alpha>0$ sufficiently large,
the strict coercivity of the sesquilinear form corresponding to (BVP) was justified in the proof
of \cite[Theorem 3.4]{SW2002} contained in \cite[Section 5]{SW2002}. The proof was based on a
suitable change of variables which reduces the $\alpha$-eigenvalue problem over 3-admissible
periodic domains to an equivalent problem over straight strips.
This together with the perturbation theory (see e.g., \cite[Chapter 7, Theorems 7.1.10, 7.1.9]{Kato}
or \cite[Chapter 8, Theorem 86]{RS1978}) and Lemma \ref{lem4.1} also implies Corollary \ref{lin-indep:cor1}.
Now, we state the discreteness of the $\mu$-eigenvalues for any fixed $k>0$ and complete the proof
of Theorem \ref{TH-phase} in Case (ii).

\begin{lem}\label{lem-x4}
Under the conditions of Proposition \ref{lem-x2.3}, the $\mu$-eigenvalues of (BVP) form at most
a discrete set in $(-1,1)$ without any accumulating point on the real axis.
\end{lem}

\begin{proof}
We carry out the proof following the ideas in the proof of \cite[Theorem 2.3]{HF97},
where the $k$-eigenvalue problem was investigated when $\mu$ is fixed.
Let $w$ be a solution to the problem (BVP).
Let $M:=\{x\in D:0<x_1<2\pi\}$ be one $2\pi$-periodic cell (see Figure \ref{case2} for the geometry of $M$)
and let $H$ be the completion of $\{\varphi\in C^1_p(\ov{M}):\varphi=0\;\textrm{on}\;\pa D\cap\ov{M}\}$
with respect to $H^1$-norm, where $C^1_p$ denotes the space of differentiable functions which are
$2\pi$-periodic with respect to $x_1$.
Note that $M$ may be disconnected. Then we can apply Green's theorem to obtain that
for any function $\psi\in H$,
\be\label{lin-indep:eq7}
\int_M \nabla w\cdot\nabla\ov{\psi}dx + \mu\int_M\left(-2ik\pa_1 w\ov{\psi}\right)dx
+(\mu^2-1)\int_M k^2 w\ov{\psi}dx=0.
\en
Let $\langle\cdot,\cdot\rangle_H$ denote the inner product of the Hilbert space $H$, which is given by
\ben
\langle \varphi,\psi\rangle_H:=\int_M\nabla \varphi\cdot\nabla\ov{\psi} dx,\quad \varphi,\psi\in H.
\enn
By Poincare's inequality, it is known that $\langle\cdot,\cdot\rangle_H$ is equivalent to the
ordinary inner product in $H^1(M)$.
Then with the aid of Riesz' representation theorem,
there exist $B,C\in \mathcal{L}(H)$ such that
\ben
\int_M\left(-2ik\pa_1 \varphi\ov{\psi}\right)dx&=&\langle B\varphi,\psi\rangle_H,\quad\varphi,\psi\in H,\\
\int_M k^2 \varphi\ov{\psi}dx&=&\langle C\varphi,\psi\rangle_H,\quad\varphi,\psi\in H,
\enn
where $\mathcal{L}(H)$ denotes the space of bounded linear operators
from $H$ into itself. Thus the formula (\ref{lin-indep:eq7}) is equivalent to the operator equation
\be\label{lin-indep:eq8}
w+\mu Bw+(\mu^2-1) Cw=0,\quad w\in H.
\en
Further, it is easily verified that $B$ and $C$ are compact operators in $\mathcal{L}(H)$.
On the other hand, let $A:\mathbb{C}\rightarrow\mathcal{L}(H)$ be an operator valued function
given by $A(\mu):=\mu B+(\mu^2-1)C$.
Then it is obvious that $A(\mu)$ is analytic in $\mathbb{C}$ and compact for each $\mu\in\mathbb{C}$.
Thus we can apply Corollary \ref{lin-indep:cor1} and the analytic Fredholm theory
(see, e.g., \cite[Theorem 8.26]{CK19}) to obtain that $(I+A(\mu))^{-1}$ exists for all $\mu\in\mathbb{C}\ba S$
where $S$ is a discrete subset of $\mathbb{C}$ with the only accumulating point at infinity.
This together with the equivalence of the problem (BVP) with the equation (\ref{lin-indep:eq8})
implies the statement of this lemma.
\end{proof}

Recall from (\ref{PBVP}) that $\sin\theta_n$ are $\mu$-eigenvalues to (BVP) for all $n\!\in\!\Z_+$. Since $\theta_n\!\in\!(-\pi/2,\pi/2)$ are distinct angles,  these $\mu$-eigenvalues must have a finite accumulating
point on the real-axis, which contradicts to Lemma \ref{lem-x4}. This implies that Case (ii) does not hold.

Finally, the relation $\G^{(1)}=\G^{(2)}$ follows by combining Case (i) and Case (ii).
This finishes the proof of Theorem \ref{TH-phase}.

We end up this section by two remarks.

\begin{rem}\label{rem0920}{\rm
By setting $u=we^{i\alpha x_1}$ with $\alpha\in \R$, the periodic boundary value problem (\ref{PBVP})
can be rewritten as
\ben
\left\{\begin{array}{lll}
\Delta u+k^2u=0\quad\mbox{in }D,\\
u=0\quad\mbox{on}\quad\pa D,\\
\mbox{$e^{-i\alpha x_1}u$ is $2\pi$-periodic with respect to $x_1$ in $D$.}
\end{array}\right.
\enn
Multiplying $\ov{u}$ on both sides of the equation and integrating over $M$, we deduce from the
quasi-periodicity of $u$ that
\ben
0=\int_{M}\left(|\nabla u|^2-k^2|u|^2\right)dx.
\enn
By Poincar\'e's inequality (see \cite[Lemma 3.13]{Monk}), it follows from the Dirichlet boundary condition
of $u$ on $\pa D\!\cap\!\ov{M}$ that $0\geq(C-k^2)\|u\|^2_{L^2(M)}$ for a constant $C>0$.
Hence, $w=e^{-i\alpha x_1}u=0$ provided $k>0$ is small enough.
Proceeding as in the proof of Lemma \ref{lem-x4}, we can conclude from the analytic Fredholm theory
(see, e.g., \cite[Theorem 8.26]{CK19}) that, for any $\alpha\in\R$, (\ref{PBVP}) admits only the trivial
solution for all $k^2\in\mathbb{C}\ba E(\alpha)$ where $E(\alpha)$ is a discrete subset of $\mathbb{C}$.
Therefore, the eigenvalues $\{K_j(\alpha)\}_{j\geq 1}$ are contained in $E(\alpha)$ and thus accumulate
only at infinity. Moreover, the associated eigenspace for each  eigenvalue $K_j(\alpha)$ is of finite
dimensions due to the compactness of corresponding operators.
}
\end{rem}

\begin{rem}{\rm
In \cite[Theorem 5.1]{CK19}, it was proved that a sound-soft scatterer can be uniquely determined by
the far-field patterns from
a finite number of incident plane waves with a fixed wave number, under the assumption that the scatterer
is contained in a ball. We note that it is interesting to extend this result to the case of periodic curves.
This may require a further investigation of properties of the $\mu$-eigenvalues with respect to domains
and is thus beyond the scope of this paper.
For analogous results with finitely many wave numbers and a fixed incident angle,
we refer to \cite[Theorem 3.2]{HF97}.
}
\end{rem}

\section{Uniqueness with phaseless data}\label{sec5}
\setcounter{equation}{0}

In contrast to the inverse problem with phase information, this section is devoted to uniqueness for
recovering the periodic curve from phaseless near-field data (i.e., Theorem \ref{TH-main} with phaseless data).
We rephrase Theorem \ref{TH-main} with phaseless data as follows.

\begin{thm}\label{TH-phaseless}
Let $\G^{(1)}$ and $\G^{(2)}$ be two periodic curves with Dirichlet boundary conditions.
Assume both of them satisfy Condition (i) or both of them satisfy Condition (ii).
Suppose that the periods of $\G^{(1)}$ and $\G^{(2)}$ are unknown.
If the corresponding phaseless total fields satisfy
\be\label{eq14}
|u_1(x;\theta_n)|=|u_2(x;\theta_n)|,\quad x\in \mathcal{D},\;n\in\Z_+,
\en
where $\{\theta_n\}_{n=1}^\infty$ are distinct incident angles in $(-\pi/2,\pi/2)$, then $\G^{(1)}\!=\!\G^{(2)}$.
Here, $\mathcal D\!\subset\!\Omega$ is a bounded domain.
\end{thm}

To prove Theorem \ref{TH-phaseless}, we will apply Rayleigh expansion (\ref{Rayleigh-up}) to show
that the phaseless near-field data corresponding to one incident plane wave uniquely determine
the total field with phase information except for a finite set of incident angles.

\begin{thm}[Phase retrieval]\label{TH-phaseless-phase}
Let $\Gamma^{(1)}$ and $\Gamma^{(2)}$ be two periodic curves satisfying the conditions in
Theorem \ref{TH-phaseless}. Assume the periods of $\Gamma^{(1)}$ and $\Gamma^{(2)}$ are $L_1>0$
and $L_2>0$, respectively.
Let $u_j(x;\theta)$ $(j=1,2)$ be the total field for the incident plane wave defined by (\ref{plane-wave})
corresponding to the periodic curve $\Gamma^{(j)}$ and let $\theta\in(-\pi/2,\pi/2)$ satisfies
$k\sin\theta L_j/\pi\notin\Z$ (i.e. $\alpha L_j/\pi\notin\Z$) for $j=1,2$.
Suppose the corresponding total fields satisfy
\be\label{z-3}
|u_1(x;\theta)|=|u_2(x;\theta)|,\quad x\in U_h,
\en
for some $h\!>\!\max\{x_2\!:x\!\in\!\G^{(1)}\!\cup\!\G^{(2)}\}$.
Then $u_1(x;\theta)\!=\!u_2(x;\theta)$, $x\!\in\!U_h$.
\end{thm}

To prove Theorem \ref{TH-phaseless-phase}, we need several auxiliary lemmata.
Let $\alpha_n$ and $\beta_n$ be defined by (\ref{beta}) with some $\theta\!\in\!(-\pi/2,\pi/2)$,
and let $\imath$ be the index for the incident plane wave (see (\ref{ui})).

\begin{lem}\label{l2}
If $\alpha L/\pi\notin\Z$, then $\alpha_n\neq-\alpha_{\imath}$ for all $n\in\Z\cup\{\imath\}$.
\end{lem}

\begin{proof}
We assume to the contrary that $\alpha_n\!=\!-\alpha_{\imath}$ for $n\!\in\!\Z\cup\{\imath\}$.
Obviously, we have $n\!\neq\!\imath$, since if otherwise there holds $\alpha_{\imath}\!=\!0$,
which contradicts $\alpha L/\pi\!\notin\!\Z$.
If $n\!\in\!\Z$ and $\alpha\!+\!n2\pi/L\!=\!-\alpha_{\imath}\!=\!-\alpha$,
we can get $\alpha L/\pi\!=\!-n\!\in\!\Z$, which also contradicts the assumption that $\alpha L/\pi\!\notin\!\Z$.
\end{proof}

In the following, we retain the notations introduced in the proof of Theorem \ref{TH-phase-period}.

\begin{lem}\label{l-n}
Suppose $\G^{(1)}$ and $\G^{(2)}$ are two grating curves with the periods $L_1>0$ and $L_2>0$, respectively.
Assume that $\alpha L_j/\pi\notin\Z$ for $j=1,2$.
Then the following statements hold.

(i) For any fixed $\tilde m\in\Z$, if
\be\label{213}
(\alpha^{(1)}_{\tilde m}-\alpha_{\imath},\beta^{(1)}_{\tilde m}-\beta_{\imath})
=(\alpha^{(2)}_{m}-\alpha^{(2)}_{n},\beta^{(2)}_{m}-\ov{\beta^{(2)}_{n}}),
\en
for some $m,n\!\in\!\Z\!\cup\!\{\imath\}$,
then $(\alpha^{(1)}_{\tilde m},\beta^{(1)}_{\tilde m})=(\alpha^{(2)}_{m},\beta^{(2)}_{m})$ and $n=\imath$.

(ii) For any fixed $\tilde m\in\Z$, if
\ben
(\alpha^{(1)}_{\tilde m}-\alpha_{\imath},\beta^{(1)}_{\tilde m}-\beta_{\imath})=(\alpha^{(1)}_{m}-\alpha^{(1)}_{n},\beta^{(1)}_{m}-\ov{\beta^{(1)}_{n}}),
\enn
for some $m,n\in\Z\cup\{\imath\}$,
then $(\alpha^{(1)}_{\tilde m},\beta^{(1)}_{\tilde m})=(\alpha^{(1)}_{m},\beta^{(1)}_{m})$ and $n=\imath$.
\end{lem}

\begin{proof}
We only prove statement (i) since statement (ii) is a consequence of statement (i) for
the special case when $\Gamma^{(1)}\!=\!\Gamma^{(2)}$.

We consider the following two cases:

\textbf{Case 1}: $\beta^{(1)}_{\tilde m}\in\R$.

Noting that $\beta^{(1)}_{\tilde m}-\beta_{\imath}>0$, we conclude from (\ref{213}) that $\beta^{(2)}_{m},\beta^{(2)}_{n}\!\in\!\R$.
Hence, the points $(\alpha_{\imath},\beta_{\imath})$, $(\alpha^{(1)}_{\tilde m},\beta^{(1)}_{\tilde m})$, $(\alpha^{(2)}_{m},\beta^{(2)}_{m})$ and
$(\alpha^{(2)}_{n},\beta^{(2)}_{n})$ are all located on the circle $x_1^2+x_2^2=k^2$ in the $x_1x_2$-plane.
From this and the relation (\ref{213}), it follows easily that there holds either
\be\label{x11}
(\alpha^{(2)}_{m},\beta^{(2)}_{m})=(\alpha^{(1)}_{\tilde m},\beta^{(1)}_{\tilde m})\;\text{and}\;(\alpha^{(2)}_{n},\beta^{(2)}_{n})=(\alpha_{\imath},\beta_{\imath})
\en
or
\be\label{x12}
(\alpha^{(2)}_{m},\beta^{(2)}_{m})=-(\alpha_{\imath},\beta_{\imath})\;\text{and}\;
(\alpha^{(2)}_{n},\beta^{(2)}_{n})=-(\alpha^{(1)}_{\tilde m},\beta^{(1)}_{\tilde m}).
\en
By Lemma \ref{l2} and the assumption $\alpha L_2/\pi\!\notin\!\Z$, the relations in (\ref{x12}) cannot be true.
Hence, the relations in (\ref{x11}) implies the desired result {\color{hw}of this lemma}.

\textbf{Case 2}: $\beta^{(1)}_{\tilde m}\notin\R$.

Observing that ${\rm Re}(\beta^{(1)}_{\tilde m}\!-\!\beta_{\imath})\!>\!0$ and
${\rm Im}(\beta^{(1)}_{\tilde m}\!-\!\beta_{\imath})\!>\!0$, we deduce from (\ref{213}) that
${\rm Re}(\beta^{(2)}_{m}\!-\!{\color{hw}\ov{\beta^{(2)}_{n}}})\!>\!0$ and
${\rm Im}(\beta^{(2)}_{m}\!-\!{\color{hw}\ov{\beta^{(2)}_{n}}})\!>\!0$.

If $\beta^{(2)}_{m}\!\notin\!\R$, then $\beta^{(2)}_{m}/i\!\in\!\R$.
This, together with ${\rm Re}(\beta^{(2)}_{m}\!-\!{\color{hw}\ov{\beta^{(2)}_{n}}})\!>\!0$, implies ${\rm Re}(-{\color{hw}\ov{\beta^{(2)}_{n}}})\!>\!0$.
This is possible only if $n=\imath$, since ${\rm Re}\,\beta_{n}^{(2)}\ge0$ for all $n\in\Z$.
Again using (\ref{213}), we find $(\alpha^{(1)}_{\tilde m},\beta^{(1)}_{\tilde m})=(\alpha^{(2)}_{m},\beta^{(2)}_{m})$, which yields the desired result of this lemma.

Now suppose that $\beta^{(2)}_{m}\!\in\!\R$, we shall derive a contradiction as follows.
Taking the real and imaginary parts of (\ref{213}) gives $\beta^{(2)}_{m}\!=\!-\beta_{\imath}$ and $\beta^{(2)}_{n}\!=\!\beta^{(1)}_{\tilde m}$.
Noting that $(\alpha^{(2)}_{m})^2\!+\!(\beta^{(2)}_{m})^2\!=\!k^2\!=\!(\alpha_{\imath})^2\!+\!(\beta_{\imath})^2$, we deduce from $\beta^{(2)}_{m}\!=\!-\beta_{\imath}$ that $|\alpha^{(2)}_{m}|\!=\!|\alpha_{\imath}|$.
Then by $\alpha L_2/\pi\!\notin\!\Z$ and Lemma \ref{l2} we obtain $\alpha^{(2)}_{m}\!=\!\alpha_{\imath}$.
Inserting this equality into (\ref{213}) gives
\be\label{213new}
\alpha^{(1)}_{\tilde m}-\alpha_{\imath}=\alpha^{(2)}_{m}-\alpha^{(2)}_{n}=\alpha_{\imath}-\alpha^{(2)}_{n}.
\en
Similarly, noting that $(\alpha^{(2)}_{n})^2\!+\!(\beta^{(2)}_{n})^2\!=\!k^2\!=\!(\alpha^{(1)}_{\tilde m})^2\!+\!(\beta^{(1)}_{\tilde m})^2$, we deduce from $\beta^{(2)}_{n}\!=\!\beta^{(1)}_{\tilde m}$ that $|\alpha^{(1)}_{\tilde m}|\!=\!|\alpha^{(2)}_{n}|$.
If $\alpha^{(1)}_{\tilde m}\!=\!\alpha^{(2)}_{n}$, then it follows from (\ref{213new}) that $\alpha^{(1)}_{\tilde m}\!=\!\alpha_{\imath}\!=\!\alpha^{(2)}_{n}$ and thus $\beta_{n}^{(2)}\!\in\!\{\pm\beta_{\imath}\}\!\subset\!\R$.
This contradicts $\beta^{(2)}_{n}\!=\!\beta^{(1)}_{\tilde m}\!\notin\!\R$.
If $\alpha^{(1)}_{\tilde m}\!=\!-\alpha^{(2)}_{n}$, then from (\ref{213new}) we deduce $\alpha_{\imath}\!=\!0$, which contradicts the {\color{red}assumption $\alpha L_2/\!\notin\!\Z$}.
The proof for Case 2 is complete.
\end{proof}

With the aid of Lemma \ref{l-n}, now we can prove Theorem \ref{TH-phaseless-phase}.

\begin{proof}[Proof of Theorem \ref{TH-phaseless-phase}]
Recalling (\ref{ui}) and (\ref{z-1}), we deduce from (\ref{z-3}) that
\be\no
&&I^{(1)}_1(x)\ov{I^{(1)}_2(x)}+I^{(1)}_2(x)\ov{I^{(1)}_1(x)}+|I^{(1)}_1(x)|^2+|I^{(1)}_2(x)|^2\\ \label{1108-1}
&&-\!I^{(2)}_1(x)\ov{I^{(2)}_2(x)}-I^{(2)}_2(x)\ov{I^{(2)}_1(x)}-|I^{(2)}_1(x)|^2-|I^{(2)}_2(x)|^2=0,\quad x\in U_h,
\en
where
\ben
I^{(j)}_1(x)=\sum_{m\in\mathcal T^{(j)}_1}\!\!\!\!A_{m}^{(j)}e^{i\alpha^{(j)}_{m}x_1+i\beta^{(j)}_{m}x_2},\quad I^{(j)}_2(x)=\sum_{n\in\mathcal T^{(j)}_2}\!\!\!\!A_{n}^{(j)}e^{i\alpha^{(j)}_{n}x_1+i\beta^{(j)}_{n}x_2}
\enn
with $\mathcal T^{(j)}_1\!:=\!\{n\!\in\!\Z\!:\!|\alpha^{(j)}_n|\!>\!k\}$ and $\mathcal T^{(j)}_2\!:=\!\{n\!\in\!\Z\!\cup\!\{\imath\}\!:\!|\alpha^{(j)}_n|\!\leq\!k\}$, $j\!=\!1,2$.

The proof can be divided into two steps as follows.

\textbf{Step 1.} We will prove that for any $\tilde m\in\mathcal T^{(1)}_2\ba\{\imath\}$ there holds
\be\label{1121-9}
\begin{cases}
A_{\tilde m}^{(1)}=A_{m}^{(2)} & \text{if there exists}\;m\!\in\!\Z\;\text{such that}\;
\alpha^{(2)}_{m}=\alpha^{(1)}_{\tilde m},\\
A_{\tilde m}^{(1)}=0 & \text{if}\;\alpha^{(2)}_{m}\neq\alpha^{(1)}_{\tilde m}\;\text{for all}\;m\in\Z,
\end{cases}
\en
and for any $\tilde m\!\in\!\mathcal T^{(2)}_2\ba\{\imath\}$ there holds
\be\label{1121-10}
\begin{cases}
A_{\tilde m}^{(2)}=A_{m}^{(1)} & \text{if there exists}\;m\!\in\!\Z\;\text{such that}\;\alpha^{(1)}_{m}\!=\!\alpha^{(2)}_{\tilde m},\\
A_{\tilde m}^{(2)}=0 & \text{if}\;\alpha^{(1)}_{m}\neq\alpha^{(2)}_{\tilde m}\;\text{for all}\;m\in\Z.
\end{cases}
\en

First, we deduce (\ref{1121-9}) for $\tilde m\in\mathcal T^{(1)}_2\ba\{\imath\}$.
Multiplying (\ref{1108-1}) by $e^{-i(\beta^{(1)}_{\tilde m}-\beta_{\imath})x_2}$ we obtain for $x\in U_h$ that
\be\label{1121-5}
\quad\;0&=&\left\{I^{(1)}_1(x)\ov{I^{(1)}_2(x)}+I^{(1)}_2(x)\ov{I^{(1)}_1(x)}+|I^{(1)}_1(x)|^2\right\}
e^{-i(\beta^{(1)}_{\tilde m}-\beta_{\imath})x_2}\\ \no
&&+\sum_{(m,n)\in\mathcal U^{(1)}_{\tilde m}}\;
A_{m}^{(1)}\ov{A_{n}^{(1)}}e^{i(\alpha^{(1)}_{m}-\alpha^{(1)}_{n})x_1}-\sum_{(m,n)\in\mathcal U^{(2)}_{\tilde m}}\;
A_{m}^{(2)}{\color{hw}\ov{A_{n}^{(2)}}}e^{i(\alpha^{(2)}_{m}-\alpha^{(2)}_{n})x_1}\\ \no
&&+\sum_{(m,n)\in(\mathcal T^{(1)}_2\times\mathcal T^{(1)}_2)\ba\mathcal U^{(1)}_{\tilde m}}\;
A_m^{(1)}\ov{A_n^{(1)}}e^{i(\alpha^{(1)}_m-\alpha^{(1)}_n)x_1+i[(\beta^{(1)}_m-\ov{\beta^{(1)}_n})
-(\beta^{(1)}_{\tilde m}-\beta_{\imath})]x_2}\\ \no
&&-\left\{I^{(2)}_1(x){\color{hw}\ov{I^{(2)}_2(x)}}\!+\!I^{(2)}_2(x)\ov{I^{(2)}_1(x)}+|I^{(2)}_1(x)|^2\right\}
e^{-i(\beta^{(1)}_{\tilde m}-\beta_{\imath})x_2}\\ \no
&&-\sum_{(m,n)\in(\mathcal T^{(2)}_2\times\mathcal T^{(2)}_2)\ba\mathcal U^{(2)}_{\tilde m}}\;
A_{m}^{(2)}\ov{A_{n}^{(2)}}e^{i(\alpha^{(2)}_{m}-\alpha^{(2)}_{n})x_1+i[(\beta^{(2)}_{m}-\ov{\beta^{(2)}_{n}})
-(\beta^{(1)}_{\tilde m}-\beta_{\imath})]x_2}
\en
where $\mathcal U^{(j)}_{\tilde m}:=\{(m,n)\in\mathcal T^{(j)}_2\times\mathcal T^{(j)}_2:\beta^{(j)}_m
-\ov{\beta^{(j)}_n}=\beta^{(1)}_{\tilde m}-\beta_{\imath}\}$, $j=1,2$.
Since $\mathcal T^{(j)}_2$ is a finite set, we know that $\mathcal U^{(j)}_{\tilde m}$ is at most
a finite set, $j=1,2$.
Using $|e^{-i(\beta^{(1)}_{\tilde m}-\beta_{\imath})x_2}|=1$, it follows from Lemma \ref{lem-x} (i) that
\ben
\left|\left\{I^{(j)}_1(x)\ov{I^{(j)}_2(x)}+I^{(j)}_2(x)\ov{I^{(j)}_1(x)}
+|I^{(j)}_1(x)|^2\right\}e^{-i(\beta^{(1)}_{\tilde m}-\beta_{\imath})x_2}\right|\leq C|I^{(j)}_1|,\quad x\in U_h,
\enn
where $C\!>\!0$ is a constant.
Thus, by similar arguments as in the proofs of (\ref{eq3}) and (\ref{eq8}),
we have $|I^{(j)}_1(x)|\!\rightarrow\!0$ as $x_2\!\rightarrow\!+\infty$ and thus
\ben\label{1121-1}
\lim_{H\rightarrow+\infty}\frac{1}{H}\int^{2H}_{H}\left\{I^{(j)}_1(x)\ov{I^{(j)}_2(x)}+I^{(j)}_2(x)\ov{I^{(j)}_1(x)}
+|I^{(j)}_1(x)|^2\right\}e^{-i(\beta^{(1)}_{\tilde m}-\beta_{\imath})x_2}dx_2=0
\enn
uniformly for all $x_1\!\in\!\R$ and $j\!=\!1,2$. Moreover, it follows easily from Lemma \ref{lem-x} (iv) that
\ben \label{1121-2}
\lim_{H\rightarrow+\infty}\frac1H\int_{H}^{2H}\!\!\!\!\!\sum_{(m,n)\in(\mathcal T^{(j)}_2\times
\mathcal T^{(j)}_2)\ba\mathcal U^{(j)}_{\tilde m}}\!\!\!\!\!A_{m}^{(j)}\ov{A_{n}^{(j)}}e^{i(\alpha^{(j)}_m-\alpha^{(j)}_n)x_1
+i[(\beta^{(j)}_{m}-\ov{\beta^{(j)}_{n}})-(\beta^{(1)}_{\tilde m}-\beta_{\imath})]x_2}dx_2=0
\enn
uniformly for all $x_1\!\in\!\R$ and $j\!=\!1,2$.
Combining (\ref{1121-5})--(\ref{1121-2}), we arrive at
\ben\label{1121-6}
\sum_{(m,n)\in\mathcal U^{(1)}_{\tilde m}}\;A_{m}^{(1)}\ov{A_{n}^{(1)}}e^{i(\alpha^{(1)}_{m}-\alpha^{(1)}_{n})x_1}
-\sum_{(m,n)\in\mathcal U^{(2)}_{\tilde m}}\;A_m^{(2)}\ov{A_n^{(2)}}e^{i(\alpha^{(2)}_m-\alpha^{(2)}_n)x_1}=0,
\quad x_1\in\R.
\enn
Similarly, multiplying (\ref{1121-6}) by $e^{-i(\alpha^{(1)}_{\tilde m}-\alpha_{\imath})x_1}$, we can employ Lemma \ref{lem-x} (iv) to obtain
\be\label{1121-7}
\sum_{(m,n)\in\mathcal V^{(1)}_{\tilde m}}\;A_{m}^{(1)}\ov{A_{n}^{(1)}}
-\sum_{(m,n)\in\mathcal V^{(2)}_{\tilde m}}\;A_{m}^{(2)}\ov{A_{n}^{(2)}}=0,
\en
where $\mathcal V^{(j)}_{\tilde m}:=\{(m,n)\in\mathcal U^{(j)}_{\tilde m}:\alpha^{(j)}_{m}-\alpha^{(j)}_{n}
=\alpha^{(1)}_{\tilde m}\!-\!\alpha_{\imath}\}$, $j\!=\!1,2$.
By Lemma \ref{l-n} we have $\mathcal V^{(1)}_{\tilde m}\!=\!\{(\tilde m,\imath)\}$
and
$\mathcal V^{(2)}_{\tilde m}=\{(m,\imath):m\in\Z\;\text{s.t.}\;\alpha^{(2)}_m=\alpha^{(1)}_{\tilde m}\}$.
Thus, noting that $\mathcal V^{(2)}_{\tilde m}$ is perhaps an empty set and
$A_{\imath}^{(1)}\!=\!1\!=\!A_{\imath}^{(2)}$, we can apply (\ref{1121-7})
to obtain that (\ref{1121-9}) holds for $\tilde m\in\mathcal T^{(1)}_2\ba\{\imath\}$.

Secondly, by interchanging the role of $|u_1(x;\theta)|$ and $|u_2(x;\theta)|$, we can employ
a similar argument as above to obtain (\ref{1121-10}) holds for
any $\tilde m\in\mathcal T^{(2)}_2\ba\{\imath\}$.

\textbf{Step 2.} We will prove that (\ref{1121-9}) holds for any $\tilde m\!\in\!\mathcal T^{(1)}_1$
and (\ref{1121-10}) holds for any $\tilde m\!\in\!\mathcal T^{(2)}_1$.

By $A_{\imath}^{(1)}\!=\!A_{\imath}^{(2)}\!=\!1$, it follows from (\ref{1108-1}) and the result in Step 1 that
\be\no
&&I^{(1)}_1(x)\ov{I^{(1)}_2(x)}+I^{(1)}_2(x)\ov{I^{(1)}_1(x)}+|I^{(1)}_1(x)|^2\\ \label{1121-11}
&&-I^{(2)}_1(x)\ov{I^{(2)}_2(x)}-I^{(2)}_2(x)\ov{I^{(2)}_1(x)}-|I^{(2)}_1(x)|^2=0,\quad x\in U_h.
\en
Let $(p_1,q_1)$ be an element in $\mathcal{B}:=\{(1,m):m\in\mathcal T^{(1)}_1\}\cup\{(2,m):m\in\mathcal T^{(2)}_1\}$
such that $|\beta^{(p_1)}_{q_1}|\!\leq\!|\beta^{(j)}_m|$ for all $(j,m)\!\in\!\mathcal{B}$.
Without loss of generality, we assume $p_1=1$.
Multiplying (\ref{1121-11}) by $e^{-i(\beta^{(1)}_{q_1}-\beta_{\imath})x_2}$ we obtain for $x\in U_h$ that
\be\label{eq9}
&&\left[I^{(1)}_1(x)e^{-i\beta^{(1)}_{q_1}x_2}\right]\left[\left(\ov{I^{(1)}_2(x)}+\ov{I^{(1)}_1(x)}\right)
e^{i\beta_{\imath}x_2}\right]\!+\!\left[I^{(1)}_2(x)e^{i\beta_{\imath}x_2}\right]\left[\ov{I^{(1)}_1(x)}
e^{-i\beta^{(1)}_{q_1}x_2}\right]\\ \nonumber
&-&\left[I^{(2)}_1(x)
e^{-i\beta^{(1)}_{q_1}x_2}\right]\left[\left(\ov{I^{(2)}_2(x)}+\ov{I^{(2)}_1(x)}\right)
e^{i\beta_{\imath}x_2}\right]\!+\!\left[I^{(2)}_2(x)e^{i\beta_{\imath}x_2}\right]\left[\ov{I^{(2)}_1(x)}
e^{-i\beta^{(1)}_{q_1}x_2}\right]\\\no
&=&0.
\en
Note that $\beta^{(j)}_m\!=\!-\ov{\beta^{(j)}_m}$ and $\left|\beta^{(1)}_{q_1}\right|\!<\!\left|\beta^{(j)}_m\!-\!\ov{\beta^{(j)}_n}\right|$
for all $m,n\!\in\!\mathcal{T}^{(j)}_1$ with $j\!=\!1,2$.
Thus, similarly to the proof of Theorem \ref{TH-phase-period}, we can apply Lemma \ref{lem-x}
to obtain that for all $j\!=\!1,2$ and $x_1\!\in\!\mathbb{R}$,
\begin{align*}
&\lim_{x_2\rightarrow+\infty}I^{(j)}_1(x)e^{-i\beta^{(1)}_{q_1}x_2}=\sum_{m\in\mathcal T^{(j)}_1~\textrm{s.t.}~ \beta^{(j)}_m=\beta^{(1)}_{q_1}} A^{(j)}_m e^{i\alpha^{(j)}_m x_1},\\
&\lim_{x_2\rightarrow+\infty}\ov{I^{(j)}_1(x)}e^{-i\beta^{(1)}_{q_1}x_2}=\sum_{n\in\mathcal T^{(j)}_1~\textrm{s.t.}~ \beta^{(j)}_n=\beta^{(1)}_{q_1}} \ov{A^{(j)}_n} e^{-i\alpha^{(j)}_n x_1},\\
&\lim_{x_2\rightarrow+\infty}\left|I^{(j)}_1(x)\right|^2 e^{-i(\beta^{(1)}_{q_1}-\beta_\imath)x_2}=0
\end{align*}
and
\begin{align*}
&\lim_{H\rightarrow+\infty}\frac{1}{H}\int^{2H}_{H}I^{(j)}_2(x)e^{i\beta_\imath x_2}dx_2
=\sum_{m\in\mathcal T^{(j)}_2~\textrm{s.t.}~\beta^{(j)}_m=-\beta_\imath}A^{(j)}_m e^{i\alpha^{(j)}_m x_1},\\
&\lim_{H\rightarrow+\infty}\frac{1}{H}\int^{2H}_{H}\ov{I^{(j)}_2(x)}e^{i\beta_\imath x_2}dx_2
=\sum_{n\in \mathcal T^{(j)}_2~\textrm{s.t.}~\beta^{(j)}_n=\beta_\imath}\ov{A^{(j)}_n} e^{-i\alpha^{(j)}_n x_1}.
\end{align*}
These together with (\ref{eq9}) imply for $x_1\in\mathbb{R}$ that
\be\label{eq10}
\sum_{(m,n)\in\mathcal{U}^{(1)}_{(1,q_1)}} A^{(1)}_m\ov{A^{(1)}_n}e^{i(\alpha^{(1)}_m-\alpha^{(1)}_n)x_1}=\sum_{(m,n)\in\mathcal{U}^{(2)}_{(1,q_1)}} A^{(2)}_m\ov{A^{(2)}_n}e^{i(\alpha^{(2)}_m-\alpha^{(2)}_n)x_1},%\quad x_1\in\mathbb{R},
\en
where $\mathcal{U}^{(j)}_{q_1}:=\{(m,n)\in\mathcal{T}^{(j)}_1\times\mathcal{T}^{(j)}_2:\beta^{(j)}_m
=\beta^{(1)}_{q_1},\beta^{(j)}_n=\beta_\imath\}\cup\{(m,n)\in\mathcal{T}^{(j)}_2\times\mathcal{T}^{(j)}_1:
\beta^{(j)}_m=-\beta_\imath,\beta^{(j)}_n=\beta^{(1)}_{q_1}\}$ for $j=1,2$.
It is clear that $\mathcal U^{(j)}_{q_1}=\{(m,n)\in(\Z\cup\{\imath\})^2:\beta^{(j)}_m-\ov{\beta^{(j)}_n}
=\beta^{(1)}_{q_1}-\beta_\imath\}$ for $j=1,2$.
Note that $\mathcal{U}^{(1)}_{q_1}$ and $\mathcal{U}^{(2)}_{q_1}$ are at most finite sets.
Then  multiplying (\ref{eq10}) by $e^{-i(\alpha^{(1)}_{q_1}-\alpha_{\imath})x_1}$, we can apply
Lemma \ref{lem-x} (iv) to obtain
\begin{align}\label{eq13}
\sum_{(m,n)\in\mathcal{V}^{(1)}_{q_1}}
A^{(1)}_m \ov{A^{(1)}_n}
=\sum_{(m,n)\in\mathcal{V}^{(2)}_{q_1}}
A^{(2)}_m \ov{A^{(2)}_n},
\end{align}
where $\mathcal{V}^{(j)}_{q_1}:=\{(m,n)\in\mathcal{U}^{(j)}_{q_1}:\alpha^{(j)}_m-\alpha^{(j)}_n
=\alpha^{(1)}_{q_1}-\alpha_\imath\}$ for $j=1,2$.
By Lemma \ref{l-n}, we have $\mathcal{V}^{(1)}_{q_1}=\{(q_1,\imath)\}$
and $\mathcal{V}^{(2)}_{q_1}\!=\!\{(m,\imath):m\in\Z\;\text{s.t.}\;\alpha^{(2)}_{m}=\alpha^{(1)}_{q_1}\}$.
Now we can apply (\ref{eq13}) and $A_{\imath}^{(1)}=1=A_{\imath}^{(2)}$ to obtain
that (\ref{1121-9}) holds for $\tilde{m}\!=\!q_1$.

To proceed further, we distinguish between the following two cases.

\textbf{Case 2.1}: there exists $q_2\in\Z$ such that $\alpha^{(2)}_{q_2}=\alpha^{(1)}_{q_1}$.
It is clear that $A^{(1)}_{q_1}\!=\!A^{(2)}_{q_2}$
and $q_2\!\in\!\mathcal{T}^{(2)}_1$, thus we have (\ref{1121-10}) holds for $\tilde{m}\!=\!q_2$.
These, together with $A_{\imath}^{(1)}\!=\!A_{\imath}^{(2)}\!=\!1$
and the result in step 1, imply that
$\widehat{I}^{(1)}_2(x)=\widehat{I}^{(2)}_2(x)$ in $x\in U_h$, where
\ben
\widehat{I}^{(j)}_2(x)=\sum_{n\in\mathcal T^{(j)}_2\cup\{q_j\}}\!\!\!\!A_{n}^{(j)}e^{i\alpha^{(j)}_{n}x_1+i\beta^{(j)}_{n}x_2},
\quad j=1,2.
\enn
Thus, it follows from  (\ref{z-3}) that
\ben
&&\widehat{I}^{(1)}_1(x)\ov{\widehat{I}^{(1)}_2(x)}+\widehat{I}^{(1)}_2(x)\ov{\widehat{I}^{(1)}_1(x)}
+|\widehat{I}^{(1)}_1(x)|^2\\
&&-\widehat{I}^{(2)}_1(x)\ov{\widehat{I}^{(2)}_2(x)}-\widehat{I}^{(2)}_2(x)\ov{\widehat{I}^{(2)}_1(x)}
-|\widehat{I}^{(2)}_1(x)|^2=0,\quad x\in U_h,
\enn
where
\ben
\widehat{I}^{(j)}_1(x)=\sum_{m\in\mathcal T^{(j)}_1\backslash\{q_j\}}\;\;
A_{m}^{(j)}e^{i\alpha^{(j)}_{m}x_1+i\beta^{(j)}_{m}x_2}, \quad j=1,2.
\enn
Let $(p_3,q_3)$
be an element in $\mathcal{C}\!:=\!\mathcal{B}\backslash\{(1,q_1),(2,q_2)\}$ s.t. $|\beta^{(p_3)}_{q_3}|\!\leq\!|\beta^{(j)}_m|$
for all $(j,m)\!\in\!\mathcal{C}$.
Then using similar arguments as above, we can obtain that (\ref{1121-9}) holds
for $\tilde{m}\!=\!q_3$ if $p_3\!=\!1$ and (\ref{1121-10}) holds for $\tilde{m}\!=\!q_3$ if $p_3\!=\!2$.

\textbf{Case 2.2}: $\alpha^{(2)}_{m}\!\neq\!\alpha^{(1)}_{q_1}$ for all $m\in\Z$. In this case,
$A^{(1)}_{q_1}=0$. Thus, similarly to Case 2.1, it follows from (\ref{z-3}) and the result in Step 1 that
\ben
&&\widehat{I}^{(1)}_1(x)\ov{I^{(1)}_2(x)}+I^{(1)}_2(x)\ov{\widehat{I}^{(1)}_1(x)}+|\widehat{I}^{(1)}_1(x)|^2\\
&&-{I}^{(2)}_1(x)\ov{{I}^{(2)}_2(x)}-I^{(2)}_2(x)\ov{I^{(2)}_1(x)}-|I^{(2)}_1(x)|^2=0,\quad x\in U_h,
\enn
where $\widehat{I}^{(1)}_1(x)$ is given as in case 2.1.
Let $(p_4,q_4)$
be an element in $\mathcal{E}\!:=\!\mathcal{B}\backslash\{(1,q_1)\}$ s.t. $|\beta^{(p_4)}_{q_4}|\le|\beta^{(j)}_m|$
for all $(j,m)\in\mathcal{E}$.
Then using similar arguments as above again, we can obtain that (\ref{1121-9}) holds
for $\tilde{m}=q_4$ if $p_4=1$ and (\ref{1121-10}) holds for $\tilde{m}=q_4$ if $p_4=2$.

For both two cases, we can repeat similar arguments again to obtain that
(\ref{1121-9}) holds for any $\tilde m\!\in\!\mathcal T^{(1)}_1$ and (\ref{1121-10}) holds for
any $\tilde m\!\in\!\mathcal T^{(2)}_1$.

Finally, noting that $A_{\imath}^{(1)}\!=\!A_{\imath}^{(2)}\!=\!1$ and combining
the results in step 1 and step 2, we have $u_1(x;\theta)\!=\!u_2(x;\theta)$ for $x\!\in\!U_h$.
\end{proof}

\begin{rem}{\rm
The proof for Theorem \ref{TH-phaseless-phase} depends only on the Rayleigh expansion (\ref{Rayleigh-up})
of the scattered fields.
Therefore, the phase retrieval result in Theorem \ref{TH-phaseless-phase} remains valid under other
boundary conditions.
}
\end{rem}

Now we are ready to prove Theorem \ref{TH-phaseless}.

\begin{proof}[Proof of Theorem \ref{TH-phaseless}]
For $j=1,2$, denote the period of the unknown grating curve $\G^{(j)}$ by $L_j>0$ and define the set $\mathcal{A}=\{\theta_n:n\in\!\Z_+~\textrm{s.t.}~k\sin\theta_nL_j/\pi\notin\Z~\textrm{for}~j=1,2\}$, where $\{\theta_n\}_{n=1}^\infty$ are the incident angles from the assumption of Theorem \ref{TH-main}.
By the analyticity of $x\mapsto|u_{j}(x;\theta)|^2$ in $\Om$ and Theorem \ref{TH-phaseless-phase},
we have $u_1(x;\theta_n)=u_2(x;\theta_n)$, $x\!\in\!U_h$, for any $\theta_n\in\mathcal{A}$.
Obviously, $\{\theta\!\in\!(-\pi/2,\pi/2):k\sin\theta L_j/\pi\in\Z~\textrm{for}~j=1,2\}$
is a finite set and thus $\mathcal{A}$ is still an infinite set.
Therefore, it follows from Theorem \ref{TH-phase} that $\Gamma^{(1)}=\Gamma^{(2)}$.
\end{proof}

\begin{rem} {\rm
Assume that the conditions presented in Theorem \ref{TH-phaseless} hold true.
Assume further that the grating periods $L_1$ and $L_2$ are known in advance and $L_1\!=\!L_2$,
then the conclusion of Theorem \ref{TH-phaseless} can be proved in a very simple way. In fact,
let $D$ be the bounded domain defined in Subsection \ref{sub-case1} if $\Gamma^{(1)}\cap\Gamma^{(2)}\neq\emptyset$
or the unbounded periodic strip defined in Subsection \ref{sub-case2} if $\Gamma^{(1)}\cap\Gamma^{(2)}=\emptyset$.
Then, due to the analyticity of the total fields and the Dirichlet boundary conditions
on $\G^{(1)}$ and $\G^{(2)}$, we can easily deduce from (\ref{eq14}) that either
$\{u_1(x;\theta_n)\}_{n\in\Z_+}$ or $\{u_2(x;\theta_n)\}_{n\in\Z_+}$
satisfy the Helmholtz equation in $D$  with wave number $k$ and vanish on $\partial D$.
This, together with the same arguments as in Section \ref{sec4}, gives that $\Gamma^{(1)}=\Gamma^{(2)}$.
}
\end{rem}

\section{Conclusion}\label{sec6}

In this paper, we have established uniqueness results for inverse diffraction grating problems for
identifying the period, location and shape of a periodic curve with Dirichlet boundary condition.
Under the a priori smoothness assumption, we proved that the unknown grating curve can be uniquely
determined by the near-field data corresponding to infinitely many
incident plane waves with different angles at a fixed wave number.
If the phase information are not available and the measurement data are taken in a bounded domain above
the grating curve, we proved that the phase information can be uniquely determined by phaseless data
provided the incident angle $\theta$ and the grating period $L$ satisfy the relation $k\sin\theta L/\pi\notin\Z$.
Our phase retrieval result (see Theorem \ref{TH-phaseless-phase}) carries over to other boundary or
transmission conditions.
However, the proof of Theorem \ref{TH-phase} for the case $\Gamma^{(1)}\cap\Gamma^{(2)}\neq\emptyset$
does not apply to the Neumann boundary condition, due to the same difficulty for inverse scattering problems
by bounded obstacles (see \cite[Page 143]{CK19} for details).
In addition, the case that $\Gamma^{(1)}\cap\Gamma^{(2)}=\emptyset $ brings extra difficulties
for treating the discreteness of the so-called $\mu$-eigenvalues in a closed waveguide.
The uniqueness with distinct incident angles for recovering penetrable gratings also remains open.
Thus it requires new mathematical theory to establish analogues of Theorem \ref{TH-phase} under
other boundary conditions.

\section*{Acknowledgments}
The authors would like to thank Professor Chunxiong Zheng from Tsinghua University
for helpful and stimulating discussions on  dispersion relations in a periodic waveguide.

%\bibliographystyle{siamplain}
%\bibliography{references}
\end{document}